\documentclass[leqno]{article}
\usepackage[margin=1.7in]{geometry}
\usepackage[hang,flushmargin]{footmisc}

\usepackage{amsfonts,url}
\usepackage{amssymb,amsthm,amsmath}
\usepackage{amstext,verbatim,graphicx,ifthen,multirow,amsthm,mathtools}
\usepackage{bm}
\usepackage{natbib}
\usepackage[bf]{caption}
\usepackage{setspace}
\usepackage{enumitem}
\usepackage{xcolor}
\usepackage{pdfsync}
\newcommand{\calA}{\mathcal{A}}

\newcommand{\Sigmacol}{\Sigma_{\varepsilon_{\cdot j}}}
\newcommand{\Sigmarow}{\Sigma_{\varepsilon_{i \cdot}}}
\newcommand{\Sigmanu}{\Sigma_{\nu}}

\newcommand{\Nu}{\mathcal N}
\newcommand{\Tau}{\mathcal T}

\newcommand{\F}{\mathcal F}

\newcommand{\htheta}{\hat\theta}
\newcommand{\stheta}{\tilde\theta}

\newcommand{\sTheta}{\Theta^{\star}}

\newcommand{\bsTheta}{\bar\Theta^{\star}}

\newcommand{\R}{\mathbb{R}}

\newtheorem{theorem}{Theorem}
\newtheorem{lemma}[theorem]{Lemma}
\newtheorem{proposition}[theorem]{Proposition}
\newtheorem{corollary}[theorem]{Corollary}

\theoremstyle{definition}
\newtheorem{remark}{Remark}
\newtheorem{example}{Example}

\newcommand{\excess}{\mathcal L}

\DeclareMathOperator{\trace}{trace}
\DeclareMathOperator*{\argmin}{argmin}
\DeclareMathOperator{\rank}{rank}
\DeclareMathOperator{\width}{w}
\DeclareMathOperator{\rad}{rad}
\DeclareMathOperator{\diam}{diam}
\DeclareMathOperator{\var}{Var}

\DeclareMathOperator{\E}{E}

\DeclarePairedDelimiter{\p}{(}{)}
\DeclarePairedDelimiter{\sqb}{[}{]}
\DeclarePairedDelimiter\set{\{}{\}}
\DeclarePairedDelimiter\norm{\lVert}{\rVert}
\DeclarePairedDelimiter\inner{\langle}{\rangle}
\DeclarePairedDelimiter\abs{\lvert}{\rvert}
\DeclarePairedDelimiter{\smallabs}{\lvert}{\rvert}
\DeclarePairedDelimiter{\floor}{\lfloor}{\rfloor}
\DeclarePairedDelimiter{\ceil}{\lceil}{\rceil}

\graphicspath{{./figures/}}
\usepackage{hyperref}

\author{David A. Hirshberg}
\title{Least Squares with Error in Variables}
\date{Stanford University}

\begin{document}
\maketitle

\begin{abstract}
Error-in-variables regression is a common ingredient in treatment effect estimators using panel data.
This includes synthetic control estimators, counterfactual time series forecasting estimators, and combinations.
We study high-dimensional least squares with correlated error-in-variables with a focus on these uses.
We use our results to derive conditions under which the synthetic control estimator 
is asymptotically unbiased and normal with estimable variance, permitting inference 
without assuming time-stationarity, unit-exchangeability, or the absence of weak factors.
These results hold in an asymptotic regime in which the number of pre-treatment periods goes to infinity
and the number of control units can be much larger $(p \gg n)$.
\end{abstract}

\section{Introduction}
Constrained least squares regression is widely used in the analysis of panel data. 
While these tend to be high dimensional regressions, often with comparable numbers of variables $p$ and equations $n$,
panel data is typically assumed to have properties that are atypical in the study of high dimensional regression.
In particular, the panel is often assumed to be the sum of a low rank matrix and a matrix of idiosyncratic noise, 
the latter being mean zero and independent unit-to-unit,
\citep[e.g.,][]{abadie2010synthetic, athey2017matrix, bai2009panel}.
This is based on the intuition that units will mostly follow a mixture of relatively few trends over time; characterized
their blend of industries, environments, cultures, etc.; with the remainder of what we observe 
approximable by idiosyncratic noise. We explore high-dimensional least squares with such data,
which does not satisfy the restricted-eigenvalue-type conditions \citep[see e.g.,][]{bickel2009simultaneous,candes2007dantzig} 
that lead to fast, sparsity-dependent rates.
Our goal is to understand two common approaches to estimating treatment effects:
those based on counterfactual time series forecasting and synthetic controls.

Concretely, we consider a setting in which we observe an $n \times p$ matrix $X = A + \varepsilon$ and
a vector $y = b + \nu$ where $A$ and $b$ are deterministic and either the rows or the columns of the $n \times (p+1)$ matrix 
$[\varepsilon, \nu]$ are independent subgaussian vectors. Our goal is to estimate a linear prediction model, i.e., $(\theta_0,\theta)$ such that 
$y \approx \theta_0 + X \theta$. This model can be used to estimate a treatment effect using panel data as follows.

\paragraph{The Counterfactual Forecasting Method.}
Suppose that the rows of $[X, y]$ represent the independent trajectories of a set of control units,
with $X$ containing the pre-treatment trajectories and $y$ outcomes during a single post-treatment period.\footnotemark 
Then we estimate a map  $x \to \hat\theta_0 + x' \hat \theta$ that predicts post-treatment outcomes, absent exposure, of a unit with the pre-treatment trajectory $x$. 
By applying this to the pre-treatment trajectory $x_e$ of an exposed unit, we impute the potential outcome $\hat y_e(0) = \hat\theta_0 + x_e' \hat \theta$ 
that we'd have observed post-treatment if, contrary to fact, it had not been exposed. The difference $\hat \tau = y_e - \hat y_e(0)$ between this 
and its observed post-treatment outcome $y_e$ estimates the effect of exposure on that unit.

\footnotetext{If we observe outcomes during multiple post-treatment periods, 
we can estimate a weighted average treatment effect over these periods by taking
$y$ to be the corresponding weighted average of post-treatment observations.}

\paragraph{The Synthetic Control Method.}
Suppose that the columns of $[X,y]$ represent the independent pre-treatment trajectories of a set of units,
with $X$ containing the trajectories of control units and $y$ that of a single exposed unit.\footnotemark
Then we estimate a map  $x \to \hat \theta_0 + x' \hat \theta$ that predicts the exposed unit's value, absent exposure, when the vector of control units takes the value $x$.
By applying this to the post-treatment values $x_e$ of the control units, we impute the potential outcome $\hat y_e(0) = \hat \theta_0 + x_e' \hat \theta'$
that we'd have observed post-treatment if, contrary to fact, it had not been exposed. 
The difference $\hat \tau = y_e - \hat y_e(0)$ between this and its observed post-treatment outcome $y_e$ estimates the effect of exposure on that unit.

\footnotetext{If we observe outcomes for multiple exposed units, 
we can estimate a weighted average treatment effect over these units by taking
$y$ to be the corresponding weighted average of pre-treatment trajectories for those units.}

\ 


Our approach to the analysis of these methods is straightforward. 
\begin{enumerate}
\item We show that the estimated map $(\htheta_0,\htheta)$ concentrates around a deterministic limit $(\stheta_0, \stheta)$.
In particular, we bound both the prediction error $\norm{A\htheta + \htheta_0 - A\stheta - \stheta_0}$
and the coefficient error $\norm{\htheta - \stheta}$.
\item Using this, we bound the deviation of our estimator $\hat \tau$ 
from an oracle estimator $\tilde\tau$ that uses this limiting map:
$\hat \tau - \tilde\tau = x_e' (\tilde\theta - \hat\theta)$. Decomposing $x_e$ into components respectively spanned by
and orthogonal to the columns of $A$, we can bound the term involving the first in terms of prediction error and the one involving the second in terms of coefficient error.
\item Finally, we characterize the oracle estimator $\tilde\tau$. As it is a simple
linear function of the exposed observations, $\tilde\tau = y_e - x_e' \tilde\theta$,
its error is straightforwardly decomposable into a deterministic bias component and a mean zero noise component.
\end{enumerate}
If the  bias component and the oracle difference $\hat \tau - \tilde \tau$ are negligible relative to the oracle noise component, 
then inference tends to be relatively straightforward. While, as usual, it is difficult to know how much bias there is,
our theory offers some heuristics.

After a brief discussion of related work, we will carry out this first step in Sections~\ref{sec:least-squares} and \ref{sec:sketch} with a fair amount of generality.
We return to treatment effect estimation in Section~\ref{sec:treatment-effects},
where we will proceed with the remaining steps for the synthetic control method.
Results for counterfactual forecasting can be derived analogously.

\section{Related Work}
\subsection{High-dimensional regression with error-in-variables}

Many generalizations of the lasso for error-in-variables regression have been studied 
in the high dimensional statistics literature \citep[e.g.,][]{datta2017cocolasso, loh2011high, rosenbaum2013improved}.
In most of this work, the signal matrix $A$ is assumed to satisfy some generalization of the restricted eigenvalue condition, i.e.,
it is assumed that $A$ is invertible as a map on approximately sparse vectors $\delta$ in the sense that
$\norm{A\delta}$ is small only if $\norm{\delta}$ is small.
When this assumption holds, these methods can achieve a 
sparsity dependent \emph{fast rate} analogous to those achieved by the lasso without error-in-variables.
The essential challenge is to find such an approximately sparse solution 
despite the sparsity-discouraging \emph{implicit} Tikhonov regularization that arises from the error in variables ---
note that $\E \norm{X\theta - y}^2 = \norm{A\theta - b}^2 + \E \norm{\varepsilon \theta - \nu}^2$,
and when $\E[\nu \mid \varepsilon] = 0$, the latter term is the Tikhonov penalty $n\norm{\Sigmarow^{1/2}\theta}^2$ 
where $\Sigmarow=\E \varepsilon' \varepsilon/n$.
The essence of the solution is to estimate and subtract off this Tikhonov penality,
for example by minimizing $\norm{X\theta - b}^2 - n\norm{\hat \Sigma \theta}^2$ for an estimate $\hat\Sigma$ of $\Sigmarow$.

However, while it is well-known that (absent error in variables) the lasso is a sensible estimator with or without a restricted eigenvalue condition, 
achieving fast rates with and an often-acceptable \emph{slow rate} without, 
these generalizations tend not to be studied without.  The same goes for the lasso itself with error in variables.
That is what we do here, studying the constrained minimizer of  $\norm{X\theta - b}^2 + n(\eta^2 - 1)\norm{\Sigmarow^{1/2} \theta}^2$ for $\eta>0$.
For $\eta$ near zero, this is least squares with the aforementioned correction for the implicit Tikhonov penalty; 
for $\eta=1$ and $\eta > 1$ this is least squares uncorrected and with additional Tikhonov regularization respectively.
Without a carefully designed signal matrix $A$, some degree of regularization is crucial: while choosing $\eta$ near zero allows us 
fast rates when there is a sparse solution and a restricted eigenvalue condition is satisfied, our slow-rate analysis 
shows that this choice is problematic when this is not the case.


\subsection{Inference in Panel Data}
While the synthetic control and counterfactual forecasting methods are widely considered sensible as point estimators,
 there is little consensus on how to do inference. Much of extant inferential theory requires
invariants like exchangeability of units, like the unit permutation test proposed by \citet{abadie2010synthetic},
or stationarity over time, like the time permutation test and t-test for bias-corrected synthetic control 
proposed by \citet{chernozhukov2017exact} and \citet{chernozhukov2018practical} respectively.
\citet{li2020statistical} characterizes the asymptotic distribution of the synthetic control 
estimator when the number of control units ($p$) is fixed and the number of time periods ($n$) goes to infinity
with comparable numbers of pre and post-treatment periods and proposes a subsampling approach to inference
that is asymptotically exact.
\citet{cattaneo2019prediction} bounds the difference between 
a synthetic control estimator and a related oracle estimator by 
a quantity that is approximated by the supremum of a certain gaussian process
for $p = o(n)$, which allows for asymptotically conservative inference when
their oracle performs well. Our treatment of the synthetic control estimator
in Section~\ref{sec:treatment-effects} is complementary, stating conditions 
under which a suitably tuned synthetic control estimator differs
negligibly from another oracle estimator. This result, 
which permits $p$ to be exponentially large relative to $n$,
allows for asymptotically exact inference when this oracle performs well. 

\section{Tikhonov-regularized least squares}
\label{sec:least-squares}
\subsection{The Estimator}

We consider the Tikhonov-regularized least squares estimator
\begin{equation}
\label{eq:ridge-general}
\begin{aligned}
 (\hat \theta_0, \hat \theta) &= \argmin_{(\theta_0, \theta) \in \Theta_0 \times \Theta} \ell(\theta_0, \theta) \quad \text{ for } \\
 \ell(\theta_0, \theta) &= \norm{\theta_0 + X\theta - y}^2 + n (\eta^2-1) \norm{\Sigmarow^{1/2} (\theta - \psi)}^2 
\end{aligned}
\end{equation}
where $\eta > 0$, $\Theta \subseteq \R^p$ is convex, $\Theta_0$ is either $\set{0}$ (no intercept) or $\R$,
 $\Sigmarow := n^{-1}\E \varepsilon' \varepsilon$ is the row covariance of the errors in covariates,
and $\psi := \argmin_{z \in \R^p} \E \norm{\nu - \varepsilon z}^2$ is the best linear predictor of the errors in the outcome from the errors in covariates.

This is a natural generalization of the least squares estimator itself. Our Tikhonov penalty is implicit in the correlation structure 
of our errors $[\varepsilon, \nu]$, arising as a term in our mean squared error; taking $\eta \neq 1$ simply scales this implicit penalty up or down.
\begin{align*}
 \E \norm{\theta_0 + X\theta - y}^2 
&= \norm{\theta_0 + A \theta - b}^2 + 2\E (\theta_0 + A\theta - b)' (\epsilon \theta - \nu) + \E \norm{\epsilon \theta - \nu}^2 \\
&= \norm{\theta_0 + A \theta - b}^2 + 0 + n\norm{\Sigmarow^{1/2} (\theta - \psi)}^2 + \E \norm{\epsilon \psi - \nu}^2.  
\end{align*}
It is feasible to do this exactly only if the autoregression vector $\psi$ is known and the covariance matrix $\Sigmarow$ is known up to a constant factor,
for example if the columns of $\varepsilon$ are iid, in which case $\psi$ would be zero and $\Sigma$ would be a multiple of the identity.
With more general covariance structures, we would need to estimate $\psi$ and $\Sigmarow$.

We characterize our estimator by showing that estimator converges 
to the deterministic model we would get by minimizing the expectation of the same loss function,
\begin{equation}
\label{eq:ridge-oracle}
\begin{aligned}
 (\tilde \theta_0, \tilde \theta) &= \argmin_{(\theta_0, \theta) \in \Theta_0 \times \Theta} \E \ell(\theta_0, \theta) \quad \text{ for } \\
     \E \ell(\theta_0, \theta)    &= \norm{\theta_0 + A\theta - b}^2 + n \eta^2 \norm{\Sigmarow^{1/2} (\theta - \psi)}^2 + \E\norm{\varepsilon \psi - \nu}^2.
\end{aligned}
\end{equation}
\noindent 
Our characterization is model-free except in the sense that we will focus on the case that $A$
is approximable by a matrix with relatively small rank. 
A concise version of our result, which holds for either independent rows or columns, follows.

\subsection{Setting}
\label{sec:setting}
We observe an $n \times p$ matrix $X = A + \varepsilon$ and
a vector $y = b + \nu$ where $A$ and $b$ are deterministic and either the rows or the columns of the $n \times (p+1)$ matrix 
$[\varepsilon, \nu]$ are independent mean-zero random vectors.
We define $\htheta$ and $\stheta$ as in \eqref{eq:ridge-general} and \eqref{eq:ridge-oracle} for $\eta \ge 0$.

We say that the noise is \emph{essentially gaussian}
if the bound $K$ defined below in \eqref{eq:more-common-notation} can be taken to be a universal constant $c$, i.e., if
the whitened independent rows or columns of $[\varepsilon,\nu]$ are $c$-subgaussian and, in the case with independent columns $\varepsilon_{\cdot j}$,
the squared norms \smash{$\norm{\varepsilon_{\cdot j}}_2^2$} concentrate around their means like a gaussian vector with the same covariance would.\footnote{This last condition rules out subgaussian vectors with non-gaussian anticoncentration behavior, like $T g$ where $T={-1, 0, 1}$ each with probability $1/3$ and $g$ is an independent gaussian vector.}
We say noise is \emph{essentially spherical} if the maximal and average eigenvalue of the column covariance matrix $\Sigmacol$ are comparable,
i.e., if $\norm{\Sigmacol}/\set{n^{-1}\trace(\Sigmacol)}$ is bounded.

\subsection{Formal Results}

\begin{theorem}
\label{theorem:rate-simplified}
Consider the setting described above with essentially gaussian and spherical noise and either the rows of $[\varepsilon, \nu]$ independent and identically distributed
or the columns of $[\varepsilon, \nu]$ independent with those of $\varepsilon$ identically distributed.
\[ \norm{\Sigmarow^{1/2}(\hat \theta -\stheta)} \le s \ \text{ and  }\ 
   \norm{(\hat \theta_0 - \tilde \theta_0) + A (\hat \theta - \tilde \theta_0)} \le \eta n^{1/2} s \]
with probability $1-c\exp\set{-cu(v,s)}$ if $s$ satisfies the fixed point condition
\begin{equation}
\label{eq:fixed-pt-simplified}
\begin{aligned}
s^2 \ge c\bigg[ &\frac{v^2 \sigma^2\width^2(\sTheta_s)}{\min(\eta_R^2,\eta_R^4) n} + \frac{(v^2 \sigma^2 \rank(A)/p_{eff})^{1,1/2}}{\eta_R^2 n} \\
    +           &\frac{v \sigma \norm{A\stheta + \stheta_0 - b} \width(\sTheta_s) + v \sigma^2 (n/p_{eff})^{1/2} \width(\sTheta_s)}{\eta_R^2 n} \bigg],  \\ 
\eta_R^2 &= \max(0,\ \eta^2 - c \rank(A)/n). 
\end{aligned}
\end{equation}
where $\sTheta_s = \set{\theta - \tilde\theta : \theta \in \Theta,\ \norm{\Sigmarow^{1/2} (\theta - \tilde \theta) \le s}}$
and $x^{1,1/2} = x + x^{1/2}$. Here $u(v,s) = \min\set{v^2\sigma^2 \width^2(\sTheta_s)/s^2,\ v^2 \rank(A),\ n}$ for $v \ge 1$.
The same holds if we substitute for $\rank(A)$ a bound on \emph{approximate rank}: any integer $R$ for which
\[ R \ge c \sigma_{R+1}(A)\width(\sTheta_s)/(s + v \sigma p_{eff}^{-1/2}). \]
\end{theorem}
Here and throughout, each instance of $c$ will denote a potentially different universal constant; $\width(S)$ can be any bound
on the gaussian width of the set $S$; and $\sigma_1(M), \sigma_2(M), \ldots$ will be the decreasing sequence of singular values of $M$ 
with $\sigma_k(M)=0$ for $k > \rank(M)$; and we make the following definitions in terms of the row and column covariance matrices
$\Sigmarow = n^{-1}\E\varepsilon' \varepsilon$ and $\Sigmacol = p^{-1}\E \varepsilon \varepsilon'$
and $\Sigmanu = \E \nu \nu'$. Furthermore, we write $\sigma$ and $p_{eff}$ to denote a noise level and effective sample size,
with definitions depending on whether the rows or columns of $\varepsilon$ are independent.
\begin{equation*}
\label{eq:common-notation}
\begin{array}{l||l|l}
                       & \text{independent rows} 
		       & \text{independent columns} \\
\hline
 \sigma^2              &  \norm{\Sigmarow} 
		       &  \norm{\Sigmacol} \\
 p_{eff}^{-1/2}        &  \norm{\Sigmarow}^{-1/2}\set{\norm{\Sigmarow^{1/2}(\stheta - \psi)} + \norm{\varepsilon_{i\cdot} \psi - \nu_i}_{L_2}}
		       &  \norm{\stheta} + \norm{\Sigmacol}^{-1/2}\norm{\Sigmanu}^{1/2} \\
\end{array}
\end{equation*}
Theorem~\ref{theorem:rate-general}, a refined version of this theorem, is stated and proven in the appendix. 
It captures behavior when noise is highly correlated and non-gaussian more precisely.

We will interpret this result for several examples of the constraint set $\Theta$. 
In this setting, the effective sample size $p_{eff}$ essentially
measures the scale of the outcome errors $\nu$ and oracle-weighted errors in covariates $\varepsilon \stheta$, whichever is larger,
in units of how many independent copies of $\varepsilon_{ij}$ we would need to average to get a random variable with the same variance.
We will assume that rank or approximate rank is not is not extremely large ($\eta_R \approx \eta$),
that we use either positive regularization or none ($\eta \ge 1$), and consider bounds for constant order $v$.

\begin{example}[The unit $\ell_1$-ball]
\label{example:l1}
When $\Theta$ is a subset of the unit $\ell_1$-ball in $\R^p$, 
the gaussian width of $\sTheta_s$ is small irrespective of radius $s$ and essentially irrespective of dimension $p$:
$\width(\sTheta_s) \le \width(\Theta) \le c\sqrt{\log(p)}$. Substituting this upper bound reduces the fixed point condition
defining $s$ to a simple inequality.
\[ s^2 \ge \frac{c}{\eta^2 n}\sqb*{ \sigma^2\log(p) + \set{\sigma \norm{A\stheta + \stheta_0 - b} + \sigma^2(n/p_{eff})^{1/2}}\sqrt{\log(p)} + 
    (\sigma^2 \rank(A)/p_{eff})^{1,1/2}}. \]
We solve this for an essentially fourth root rate, $s = c \sigma \eta^{-1} (n \cdot p_{eff})^{-1/4}\sqrt{\log(p)}$,
if the oracle weights fit reasonably well and the signal is not too complex in the sense that
\smash{$\norm{A\stheta + \stheta_0 - b} \le c\sigma (n/p_{eff})^{1/2}$} and 
\smash{$\rank(A) \le \min\set{(n \cdot p_{eff})^{1/2}\sqrt{\log(p)}, \sigma^2 n \log(p)}$}.
\end{example}

\begin{example}[Euclidean space $\R^p$]
\label{example:euclidean}
When $\Theta \subseteq \R^p$, $\width(\sTheta_s) \le c s\sqrt{p}$. Substituting this bound, we characterize our rate of convergence via the following
fixed point condition.
\[ 
s^2 \ge \frac{c}{\eta^2 n} \sqb*{ \sigma^2 s^2 p + \set{\sigma \norm{A\stheta + \stheta_0 - b} + \sigma^2 (n/p_{eff})^{1/2}}sp^{1/2}
+ \set{\sigma^2 \rank(A)/p_{eff}}^{1,1/2}}. \]
If $p \ge c\eta^2 n/ \sigma^2$, this is not solvable, so our theorem implies no bound. 
Otherwise, we get a variant of the familiar
$\sqrt{p/n}$ rate, $s = c\sigma^2\eta^{-2} \sqrt{p / (n \cdot p_{eff})}$, 
if the oracle weights fit reasonably well and the signal is not too complex in the sense that
$\norm{A\stheta + \stheta_0 - b} \le c\sigma (n/p_{eff})^{1/2}$ and 
$\rank(A) \le \min(\eta^{-2}p, \ \sigma^2  p_{eff}^{-1} \eta^{-4}p^2)$.
\end{example}

\begin{remark}
\label{remark:low-noise}
The discussion of examples aboves focuses on the high-noise regime, 
in which a term proportional to $\sigma^2$ dominates the right side of our fixed point conditions above.
In the low noise regime, in which $\sigma$ is close to zero, behavior is different:
for small enough $\sigma$, the fixed point condition from Example~\ref{example:euclidean} is approximated by the simplified condition
$s^2 \ge (c/\eta^2 n)\set{\sigma^2 \rank(A)/p_{eff}}^{1/2}$,
which implies a near-$n^{-1/2}$ rate with weak dependence on rank.
\end{remark}

\begin{remark}
\label{remark:rank-dependence}
These bounds do not require that $A$ be exactly rank-degenerate,
as we can substitute for $\rank(A)$ an approximate rank as described in Theorem~\ref{theorem:rate-simplified}.
For the rank bound discussed in Example~\ref{example:l1} to hold
in the high-noise regime $\sigma^2 \ge (p_{eff} / n)^{1/2}$,
it suffices that \smash{$c\sigma_{\ceil{(n \cdot p_{eff})^{1/2}}}(A) \le \sigma n^{1/2}$}.
For comparison, in the case that $p$ and $n$ are comparable, 
the largest singular value of $\varepsilon$
will be on the order of \smash{$\sigma n^{1/2}$}, so this allows $A$ to have 
\smash{$(n \cdot p_{eff})^{1/2}$} principal components that are `visible' above the noise level
and then arbitrarily many `blending in' at or below it.
\end{remark}

To develop intuition for this result, we will now sketch its proof. The complete proof is deferred to the appendix.

\section{Proof Sketch}
\label{sec:sketch}

We sketch our argument in the simplified case of regression without intercept, i.e., taking $\theta_0 = 0$ above. 
We begin by considering the first order optimality conditions for $\tilde \theta$, which imply that
\begin{equation}
\label{eq:projection-theorem-sketch}
 0 \le (A\tilde\theta - b)'A(\theta - \tilde\theta) + n\eta^2 (\tilde\theta - \psi) \Sigmarow (\theta - \tilde\theta) \ \text{ for all }\ \theta \in \Theta.
\end{equation}
This generalizes the more familiar result for least squares projections $(\eta = 0)$, in which we have only the first term on the right,
to the case of Tikhonov regularized projections \citep[see Sections 3.3-3.5]{peypouquet2015convex}.

To characterize $\htheta$, we use the property that 
$\excess(\htheta - \stheta) \le 0$ for $\excess(\delta) := \ell(\delta + \stheta) - \ell(\stheta)$
and therefore, if $\excess(\delta) > 0$ except on some set $\Theta^{\star}$, then $\htheta - \stheta \in \Theta^{\star}$.
Thus, taking $\sTheta$ to be a set on which bounds like those in Theorem~\ref{theorem:rate-simplified} hold,
it suffices to bound $\excess$ above zero on its complement. We begin with some arithmetic.
\begin{equation}
\label{eq:basic-ineq-sketch}
\begin{aligned}
\excess(\delta) &= \norm{X(\delta + \stheta) - y}^2 -  \norm{X\stheta - y}^2 \\ 
&+ n (\eta^2-1) (\norm{\Sigmarow^{1/2}(\delta + \stheta-\psi)}^2 - \norm{\Sigmarow^{1/2}(\stheta-\psi)}^2)  \\
  &=  \norm{X\delta}^2 + n(\eta^2-1) \norm{\Sigmarow^{1/2}\delta}^2 + 2[(X\stheta - y)' X \delta  + n(\eta^2-1) (\stheta-\psi)'\Sigmarow \delta]  \\
  &= \norm{X\delta}^2 +  n(\eta^2-1) \norm{\Sigmarow^{1/2}\delta}^2 + 2[ (A\stheta - b)' A \delta + n\eta^2 (\stheta - \psi)'\Sigmarow \delta ] \\
&+ 2(\varepsilon\stheta - \nu)' A \delta + 2(A\stheta - b)' \varepsilon \delta  
  + 2(\varepsilon\psi - \nu)' \varepsilon \delta  + 2 (\stheta - \psi)'(\varepsilon'\varepsilon - n\Sigmarow)\delta.
\end{aligned}
\end{equation}
The last expression uses a fairly elaborate decomposition of the bracketed term in the preceding expression.
In this decomposition, the bracketed term is nonnegative via \eqref{eq:projection-theorem-sketch}
and all terms in the second line have mean zero --- the penultimate term because $\varepsilon\psi$ is the 
orthogonal projection of $\nu$ on the subspace $\set{ \varepsilon x : x \in \R^p}$. 
Furthermore, as this implies that $(\varepsilon\psi - \nu)' \varepsilon \delta  + (\stheta - \psi)'(\varepsilon'\varepsilon)\delta = 
(\varepsilon \stheta - \nu)'\varepsilon \delta$ has mean $(\stheta - \psi)'n\Sigmarow\delta$,
we can write the sum of the last two terms as $2[(\varepsilon\stheta - \nu)' \varepsilon - \E (\varepsilon\stheta - \nu)' \varepsilon] \delta$.
Thus, 
\begin{equation}
\label{eq:excess-sketch}
\begin{aligned} 
\excess(\delta) &\ge \norm{X\delta}^2 +  n(\eta^2-1) \norm{\Sigmarow^{1/2}\delta}^2 \\
&+  2(\varepsilon\stheta - \nu)' A \delta + 2(A\stheta - b)' \varepsilon \delta 
 + 2[(\varepsilon\stheta - \nu)' \varepsilon - \E (\varepsilon\stheta - \nu)' \varepsilon] \delta.
\end{aligned}
\end{equation}
As \smash{$\norm{X\delta}^2$} concentrates around its mean \smash{$\E \norm{X\delta}^2 = \norm{A \delta}^2 + n\norm{\Sigmarow^{1/2}\delta}^2$,} 
roughly speaking $\excess(\delta) > 0$ when the mean \smash{$\norm{A \delta}^2 + n\eta^2\norm{\Sigmarow^{1/2}\delta}^2$} of the first line 
exceeds the noise process on the second. Thus, we will show that with high probability, $\excess(\delta) \ge 0$ except for $\delta$ in 
a set \smash{$\sTheta=\set{ \delta : \norm{A\delta} \le r, \norm{\Sigmarow^{1/2}\delta} \le s}$} on which the first line's mean is small.

The first of the three noise terms on the second line of \eqref{eq:excess-sketch} is qualitatively similar to what we see without
error in covariates (when $\varepsilon = 0$). It will be small uniformly over $\delta$ in a set of local deviations $\sTheta$
if the image $A\sTheta$ is relatively small, e.g., when it is low dimensional because $A$ is low rank.
The second will be small when the oracle prediction error $\norm{A\stheta - b}$ is small, 
which will again tend to be the case when $A$ is low rank;
in this case, the subspace of solutions $A\theta = b$ is large,
so the constraint $\theta \in \Theta$ and the Tikhonov penalty tend not to push $\stheta$ away from it.
The third, on the other hand, is unrelated to $A$. For intuition about its behavior, 
it is useful to consider a \emph{decoupled} variant, $2(\varepsilon\stheta - \nu)' \tilde \varepsilon\delta$,
in which $\tilde \varepsilon$ is distributed like $\varepsilon$ but independent of $\varepsilon$ and  $\nu$.
If noise were gaussian, this term and our decoupled variant satisfy essentially the same tail bounds \citep[Theoreom 4.2.7]{de2012decoupling},
and for our purposes they behave similarly enough in the essentially gaussian case as well. 
Then conditional on $\varepsilon$ and $\nu$, this would be the inner product  
$2z'\Sigmarow^{1/2}\delta$ where \smash{$z = (\varepsilon\stheta - \nu)' \tilde \varepsilon\ \Sigmarow^{-1/2}$} is a subgaussian vector.
When the rows of $[\varepsilon,\nu]$ are iid and gaussian, this has a particularly simple form:
the right side of \eqref{eq:excess-sketch} is approximately
\begin{equation}
\label{eq:excess-sketch-approx}
\begin{aligned} 
\tilde \excess(\delta) &:= \norm{A\delta}^2 +  n\eta^2 \norm{\Sigmarow^{1/2}\delta}^2 \\ 
&+  2V_{eff}^{1/2} g_1' A \delta + 2\norm{A\stheta - b}g_2' \Sigmarow^{1/2} \delta + 2  V_{eff}^{1/2} n^{1/2} g_3' \Sigmarow^{1/2}\delta 
\end{aligned}
\end{equation}
where $V_{eff} = \E\norm {(\varepsilon_{i\cdot}\stheta - \nu_i)}^2$ and $g_1 \ldots g_3 \in \R^n$ are standard gaussian vectors,
\begin{align*}
g_1 &= V_{eff}^{-1/2}(\varepsilon\stheta - \nu), \\
g_2 &= [(A\stheta - b)'/\norm{A\stheta - b}] \varepsilon \Sigmarow^{-1/2}, \\ 
g_3 &\approx (\E\norm{\varepsilon \stheta - \nu}^2)^{-1/2} z = V_{eff}^{-1/2} n^{-1/2} z.
\end{align*}
For the remainder of our sketch, we'll work with $\tilde \excess(\delta)$ 
as if it were a uniform lower bound on $\excess(\delta)$, showing that
$\tilde\excess(\delta) > 0$ unless $\delta \in \sTheta$.
While the details are slightly different when $[\varepsilon, \nu]$ is not gaussian 
and/or has independent columns rather than rows,  
the intuition we'll develop here is essentially valid in the general case we consider.

\subsection{Bounding $\norm{\Sigmarow^{1/2}(\htheta - \stheta)}$}

If there were no signal ($A=0,b=0$) then the condition $\tilde \excess(\delta) > 0$ 
would reduce to a simple ratio process bound,
\[ -2V_{eff}^{1/2}n^{1/2} g_3'\Sigmarow^{1/2}\delta\ /\ \norm{\Sigmarow^{1/2}\delta}^2 <   n \eta^2. \]
Thus, the possibility that $\norm{\Sigmarow^{1/2}(\hat \theta - \theta)}=s$ is ruled out if
\[ 2V_{eff}^{1/2}n^{1/2}\sup_{\delta \in \Theta^{\circ}_s} -g_3' \Sigmarow^{1/2}\delta <    n \eta^2 s^2 \ \text{ where }\ 
\Theta^{\circ}_s := \set{ \theta - \stheta : \theta \in \Theta, \norm{\Sigmarow^{1/2}(\theta - \stheta)} = s}.
 \]
As the supremum on the left side will concentrate around its mean, the gaussian width \smash{$\width(\Sigmarow^{1/2}\Theta^{\circ}_s)$}, 
we should expect to rule out this possibility for all 
\smash{$s > s^{\star} := \min\{s > 0 :$}\ \smash{ $2V_{eff}^{1/2} \width(\Theta^{\circ}_s) \le n^{1/2} \eta^2 s^2 \}$.}
Thus, we should expect that \smash{$\htheta - \stheta \in \Theta^{\star}_{s^{\star}} := \cup_{s \le s^{\star}} \Theta_s^{\circ}$}. 
This is essentially what we see in Theorem~\ref{theorem:rate-simplified}.

The addition of signal, unless it is carefully designed, does little to improve this rate of convergence $s^{\star}$.
The reason for this is that the term $\norm{A\delta}^2$ helps rule out the possibility that $\tilde\excess(\delta) \le 0$ 
for certain deviations $\delta$, but does not for others, e.g. those within the kernel of $A$ or more generally in subspaces on which $A$ has small operator norm.
If the aforementioned ratio process approximately achieves its maximum for one of these deviations that is not ruled out, 
then the implied rate $s^{\star}$ does not improve. While the design of signals that drive convergence of $\Sigmarow^{1/2}\htheta$ to $\Sigmarow^{1/2}\theta$,
e.g. those satisfying a restricted eigenvalue condition, 
is a widely studied topic in high dimensional statistics, we will not consider signals with this property here.
Thus, when we are interested in convergence in this sense, the presence of signal typically
just saddles us with additional noise.

We'll now consider the negative influence of the noise terms in \eqref{eq:excess-sketch-approx}
when there is signal. We'll show that, by allocating fractions of our second signal term \smash{$\eta^2 n \norm{\Sigmarow^{1/2}\delta}^2$},
we can cancel out the potentially negative influence of each noise term for all $\theta \in \Theta_s^\circ$.
Call these fractions $\alpha_1\ldots\alpha_3$ (with $\alpha_1 + \alpha_2 + \alpha_3 = 1$).
\begin{enumerate}
\item Consider the sum of the first signal term and first noise term,
\smash{$\norm{A\delta}^2 + 2V_{eff}^{1/2} g_1' A \delta$},
for $\delta$ with $\norm{A\delta}=r$. This is no smaller than
$r^2 - 2V_{eff}^{1/2}\sup_{\delta} -g_1' A \delta$ for $\delta$ ranging 
over the set of $\delta \in \Theta^\circ_s$ with $\norm{A\delta}=r$.
Moreover, this supremum concentrates around its mean,
the gaussian width of the set $\set{A\delta : \delta \in \Theta^\circ_s, \norm{A\delta}=r}$.
This is no larger than $cr\rank^{1/2}(A)$, as it is contained in the ball of radius $r$ in 
the $\rank(A)$-dimensional image of $A$. Thus, this sum is essentially at least 
\[ \min_{r} r^2 -  2cV_{eff}^{1/2}\rank^{1/2}(A) r = -[cV_{eff}^{1/2}\rank^{1/2}(A)]^2, \]
and it is cancelled if $\alpha_1 n\eta^2 s^2$ exceeds the magnitude of this quantity.
\item The second noise term is cancelled if 
\smash{$2\norm{A\stheta - b} \sup_{\delta \in \Theta^{\circ}_s} -g_3'\Sigmarow^{1/2}\delta \le$} $\alpha_2n\eta^2 s^2$.
And as the supremum here concentrates around its mean, the gaussian width of \smash{$\Sigmarow^{1/2}\Theta^\circ_s$},
this is equivalent to the condition \smash{$ 2\norm{A\stheta - b} \width(\Sigmarow^{1/2}\Theta^\circ_s)$} $\le \alpha_2 n\eta^2 s^2$.
\item
The third noise term, discussed in the no-signal case,
is cancelled by its fraction of signal if $2V_{eff}^{1/2}n^{1/2}\width(\Sigmarow^{1/2}\Theta^{\circ}_s) \le \alpha_3 n\eta^2 s^2$.
\end{enumerate}
Summing, the negative influence of our noise terms is essentially canceled if 
\begin{equation}
\label{eq:sketch-s-fp} 
n\eta^2 s^2 \ge cV_{eff} \rank(A)
+2\norm{A\stheta - b}\width(\Sigmarow^{1/2}\Theta^\circ_s) 
+2V_{eff}^{1/2}n^{1/2}\width(\Sigmarow^{1/2}\Theta^{\circ}_s). 
\end{equation}
As $V_{eff} \approx \sigma^2/p_{eff}$, this is essentially the fixed point condition from Theorem~\ref{theorem:rate-simplified}.\footnote{It is an even better match for the refined fixed point 
condition in Theorem~\ref{theorem:rate-general}. Differences between this condition and Theorem~\ref{theorem:rate-simplified} occur mostly where we've simplified
Theorem~\ref{theorem:rate-general} for interpretability.} 

Our formal proof differs largely in that it has to deal with two source of additional complexity.
The first is the difference between $\norm{X\delta}$ and  $\E \norm{X \delta}$, which we've ignored above.
This is done using techniques relatively similar to those we've used above.
The second is showing that these conclusions, which we've talked about for one $(s,r)$ pair,
hold simultaneously for all (s,r) pairs on a high probability event. To do this,  
show that it suffices that they hold for a single carefully-chosen pair.
This argument exploits the convexity of $\Theta$,
charactering the terms away from this critical pair 
by considering the way they scale when we replace $\delta$ with a scalar multiple $\alpha\delta$.

\subsection{Bounding $\norm{A(\htheta - \stheta)}$}
In this section, we characterize prediction error, i.e., the convergence of $A\htheta \to A\stheta$.
In the previous section, we showed \smash{$\tilde \excess(\delta) > 0$} 
unless \smash{$\norm{\Sigmarow^{1/2} \delta} \le s$}; here we will strengthen this
by adding the condition that $\norm{A\delta} \le r$ for $r=\eta n^{1/2}s$. In particular,
we will rule out the possibility that \smash{$\tilde \excess(\delta) > 0$} 
for $\delta$ with $\norm{\Sigmarow^{1/2} \delta} \le s$ and $\norm{A\delta} = r$.
Bounding terms in $\tilde\excess(\delta)$ essentially as in the previous section, 
$\tilde \excess(\delta) > 0$ for such all $\delta$ if 
\begin{equation}
\label{eq:sketch-r-fp}
r^2 >  2c V_{eff}^{1/2} \rank^{1/2}(A)r + 2\norm{A\stheta - b}\width(\Sigmarow^{1/2}\Theta^{\circ}_{s}) +  V_{eff}^{1/2} n^{1/2}\width(\Sigmarow^{1/2}\Theta^{\circ}_{s}). 
\end{equation}
This holds for $r$ exceeding the larger root of this quadratic. Up to constant factors, this is $\eta n^{1/2}s$ 
for $s$ satisfying the aforementioned lower bound \eqref{eq:sketch-s-fp} with equality.


\begin{remark}
\label{rema:non-identical}
While we assume that either rows or columns of $\varepsilon$ are identically distributed for simplicity, 
we have tried to define some quantities in ways that are informative about what happens more generally.
In particular, our definition of $\Sigmarow$ and $\psi$ do not rely on identical distribution,
nor does our derivation of the loss lower bound \eqref{eq:excess-sketch}. Thus, even without identical distribution we should expect
that the least squares regression coefficents $\hat \theta$ (\ref{eq:ridge-general} with $\eta=1$)
converge to the autoregression coefficients $\psi$.
\end{remark}

We use the bounds derived here to characterize the synthetic control estimator.

\section{The Synthetic Control Estimator}
\label{sec:treatment-effects}

\subsection{Setting}
\label{sec:sc-setting}
To discuss the synthetic control estimator, we'll have to extend our model to include post-treatment outcomes.
In this extended model, we observe an $(n+1) \times (p+1)$ matrix with independent columns, each a time series for a different unit.
The first $p$ units are controls and the last unit is exposed to treatment only in the last time period. As before,
we call the $n \times (p+1)$ submatrix of pre-treatment outcomes $[X\ y]$, and we let $[x_e \ y_e]$ 
be outcomes for the final period. And we assume that treatment affects only the exposed unit 
and does so only after exposure, i.e., that the only difference in what we'd have observed if
treatment had not occurred would be the substitution of a control potential outcome $y_e(0)$ 
for $y_e$. Our goal is to estimate the effect of treatment on the exposed observation, $\tau = y_e - y_e(0)$,
having observed
\begin{equation}
\label{eq:synthdid-model}
 \begin{pmatrix} X   & y \\ x_e' & y_e \end{pmatrix}  
  = \begin{pmatrix} X & y \\ x_e' & y_e(0)+\tau \end{pmatrix}
  = \begin{pmatrix} A & b \\ a_e' & b_e \end{pmatrix} 
  + \begin{pmatrix} 0 & 0 \\ 0 & \tau \end{pmatrix}
  + \begin{pmatrix} \varepsilon & \nu \\ \varepsilon_e' & \nu_e \end{pmatrix}.
\end{equation}
In the last expression, we decompose our observations into a `sytematic component', a `noise component', and the effect of treatment.
We assume that the elements $A$, $b$, $a_e$, and $b_e$ of the systematic componenent are deterministic; 
the columns of $[\varepsilon;\ \varepsilon_{e}]$ of the noise component are 
identically distributed with mean zero; and, for simplicity, that these columns are essentially spherical and gaussian
and that $[\nu;\ \nu_e]$ is distributed like \smash{$p_e^{-1/2} [\varepsilon;\ \varepsilon_e]$}
for $p_e \in \R$.

This can be interpreted as a model for $p+p_{e}$ units with iid noise vectors, 
in which the treated unit's time series $[y;y_e]$ is the average of $p_e \in \mathbb{N}$ independent time series:
\begin{equation}
\label{eq:synthdid-disaggregated}
\begin{aligned} 
\begin{pmatrix} y \\ y_e \end{pmatrix} 
= \frac{1}{p_e}\sum_{j=p+1}^{p+p_e} \begin{pmatrix} Y_{\cdot j} \\ Y_{ej} \end{pmatrix}  
= \frac{1}{p_e}\sum_{j=p+1}^{p+p_e} \sqb*{ \begin{pmatrix} B_{\cdot j} \\ B_{ej} \end{pmatrix} +
				       \begin{pmatrix} 0 \\  \Tau_j          \end{pmatrix} +
				       \begin{pmatrix} \Nu_{\cdot j} \\  \Nu_{ej} \end{pmatrix} }
\end{aligned}
\end{equation}
where the noise vectors $[\Nu_{\cdot j};\ \Nu_{ej}]$ are independent, independent of $[\varepsilon;\ \varepsilon_e]$,
and distributed like the columns of $[\varepsilon;\ \varepsilon_e]$.
In this interpretation, $[b;\ b_e]$, $\tau$, and $[\nu;\ \nu_e]$ are averages of the corresponding terms in \eqref{eq:synthdid-disaggregated}.
While the use of these aggregates is sufficient for point estimation, 
access to the individual series \smash{$[Y_{\cdot j};\ Y_{ej}]$} opens up a number of possibilities
for inference. 
 

It is common to assume that the systematic component is a low rank matrix, and interpret its singular value decomposition
$\sum_{k}\sigma_k u_k v_k'$ as describing a set of common trends over time --- the scaled left singular vectors $\sigma_k u_k$ ---
that are reflected in the $i$th unit according to the `mixing weight' $(v_k)_i$ \citep[e.g.,][]{abadie2010synthetic, athey2017matrix, bai2009panel}.
When these units are regions, as is common, these mixing weights are interpreted as reflecting their mix of industries, cultures, environments, etc.
We will not assume low rank, but we will discuss the impact of low rank on our results.

We define $\sigma_e^2$ to be the variance of an element of $\varepsilon_e$ and $\sigma^2=\norm{\Sigmarow}$.\footnote{This definition of $\sigma^2$ differs from the definition $\sigma^2=\norm{\Sigmacol}$ used in Theorem~\ref{theorem:rate-simplified}. Nonetheless, we can use Theorem~\ref{theorem:rate-simplified} as written, changing only unspecified constant factors, with the definition we use here. The two definitions are equivalent up to constant factors, 
as \smash{$\Sigmarow=n^{-1}\trace(\Sigmacol) I$} \eqref{eq:recall-common-notation}
and we have assumed that \smash{$\norm{\Sigmacol} \le c n^{-1}\trace(\Sigmacol)$}, i.e., noise is essentially spherical.}

Throughout, we will consider regression without intercept, taking $\Theta_0=\set{0}$.
\subsection{Summary}
\label{sec:sc-summary}
The synthetic control approach of \citet{abadie2010synthetic} is based on the principle that if
the trajectory of a linear combination of control units approximates the treated unit before onset of treatment,
the same combination should approximate the outcome we'd observe after onset if, contrary to fact, it were not treated.
In the notation above, we can write this as $X\theta \approx y \implies x_e' \theta \approx y_e(0)$,
which suggests the synthetic control treatment effect estimator $\hat \tau = y_e - x_e' \hat \theta$ where $X\hat\theta \approx y$. 
\citet{abadie2010synthetic} chooses $\hat\theta$ by least squares from the set $\Theta$ of nonnegative weights summing to one; 
\citet{doudchenko2016balancing} suggests the inclusion of a Tikhonov penalty as in \eqref{eq:ridge-general}.

The analysis of this estimator is complicated by the use of weights fit to these noisy pre-treatment observations,
which allows the possibility of overfitting and bias arising from dependence between $\hat\theta$ and $[\varepsilon_e', \nu_e]$.
To simplify this, we approximate the behavior of 
$\hat \tau$ by a variant with deterministic weights, $\tilde \tau = y_e - x_e' \tilde\theta$ for $\tilde\theta$ as in \eqref{eq:ridge-oracle},
and analyze deviation from the oracle $\hat \tau - \tilde \tau$ and oracle error $\tilde \tau - \tau$ separately.
Because the oracle weights $\tilde\theta$ are deterministic, the error of the oracle estimator 
decomposes cleanly into a bias term and a noise term, yielding a three-term error decomposition.
\begin{equation}
\label{eq:three-term-error}
\hat \tau - \tau = (\hat \tau - \tilde \tau) +
 \underbrace{(b_e - a_e'\tilde\theta)}_{\E\tilde\tau - \tau} + 
\underbrace{(\nu_e - \varepsilon_e'\tilde\theta)}_{\tilde\tau - \E\tilde\tau}.
\end{equation}
When the first two terms are negligible, this error behaves like the third: an average of independent mean-zero random variables.
And when our treated unit $[y; y_e]$ is an aggregate of many treated units in the sense of \eqref{eq:synthdid-disaggregated}, 
this makes inference straightforward, as we can use resampling methods like the bootstrap or jackknife appropriate for such averages.
In Section~\ref{sec:sc-deviation} below, we bound the first term and give sufficient conditions for it to be negligible. In Section~\ref{sec:oracle-behavior}, we discuss the second term.

In this context, negligible means vanishingly small relative to the standard deviation of the third term: \smash{$\sigma_e p_{eff}^{-1/2}$} where $\sigma_e^2 = \var(\varepsilon_{ej})$ and
the effective sample size parameter $p_{eff}$, defined \smash{$1/p_{eff} = 1/p_e + \norm{\stheta}^2$},
is roughly the smaller of (i) $p_e$, the number of treated units 
in our aggregate model \eqref{eq:synthdid-disaggregated} and (ii) $p_c = 1/\norm{\stheta}^2$, which is roughly the number of strongly-weighted control units.\footnote{If all nonzero synthetic control weights $\tilde\theta_i$ were equal, $p_c$ would be the number of controls with nonzero weight. This new definition of $p_{eff}$ is equivalent up to a factor of two to the one used in Theorem~\ref{theorem:rate-simplified},
as $\Sigma_{\nu}=\Sigmacol / p_e$ in our setting.
We will use the new one when applying Theorem~\ref{theorem:rate-simplified}.}

\subsection{Deviation from the Oracle}
\label{sec:sc-deviation}

The first result is stated in terms of the singular value decomposition $A=\sum_k \sigma_k u_k v_k'$;
$\varepsilon_{ej}$ and $\varepsilon_{\cdot j}$, the $j$th element of $\varepsilon_e$ and $j$th column of $\varepsilon$.
\begin{proposition}
\label{proposition:sc-bias-bound}
In the setting described in Section~\ref{sec:sc-setting}, let $\hat \tau = y_e - x_e' \hat \theta$ and $\tilde \tau = y_e - x_e' \tilde\theta$ for $\hat \theta$ and $\tilde \theta$ solving \eqref{eq:ridge-general} and \eqref{eq:ridge-oracle} respectively with any convex
set $\Theta$ and regularization parameter $\eta \ge 0$. For
$s$ and $u(s,v)$ as in Theorem~\ref{theorem:rate-simplified} and any $(w_1,w_2)$,
on an event of probability $1-c\exp\set{-cu(s,v)} - 2\exp(cw_1^2 \width^2(\sTheta_s)/s^2) - 2\exp(-w_2^2)$,
\begin{equation}
\label{eq:sc-deviation}
\begin{aligned}
\abs{\hat \tau - \tilde \tau}
&\le (s/\sigma)(\sqrt{2}D + w_2 \norm{\varepsilon_{ej} - \psi_{col}'\varepsilon_{\cdot j}}_{L_2}) + (1+ w_1)\norm{\psi_{col}'\Sigmacol^{1/2}} \width(\sTheta_s), \\
D^2 &= \sum_k \frac{ (a_e' v_k)^2 }{1+\sigma_k^2/(\sigma^2\eta^2 n)},\  \psi_{col} = \argmin_{v} \E \norm{\psi' \varepsilon - \varepsilon_e}^2, \notag \\ 
\end{aligned}
\end{equation}
\end{proposition}

We summarize by stating conditions for approximate normality around the mean of the oracle estimator.
\begin{corollary}
\label{corollary:sc-normality}
In the setting of Section~\ref{sec:treatment-effects}, consider the synthetic control estimator
$\hat \tau = y_e - x_e' \htheta$ for $\htheta$ satisfying \eqref{eq:ridge-general} with $\Theta=\set{\theta \in \R^p : \theta_i \ge 0 \text{ for all } i,\ \sum_{i=1}^p \theta=1}$. 
The $z$-statistic $(\hat \tau - \E\tilde\tau)/\sigma_{\tau}$ for $\sigma_{\tau} = \sigma_e p_{eff}^{-1/2}$ 
converges in law to a standard normal if
\begin{enumerate}
\item $n^{-1}\norm{A\tilde\theta - b}^2 \ll n^{\kappa-1} \sigma_{\tau}^2$ for any $\kappa \in \mathbb{R}$;
\item $\eta$ satisfies usually-almost-vacuous lower bounds \eqref{eq:near-vacuous-eta-bounds} stated in the Appendix;
\item 
\[ \rank(A) \ll \min\set*{ n \p*{\frac{\eta \sigma_e}{\tilde D}}^2,\ \frac{\sigma^2 n^2}{p_{eff}} \p*{\frac{\eta \sigma_e}{\tilde D}}^4 }; \]
\item
\begin{align*} 
p_{eff} \ll \min\bigg[&\frac{n}{\log(p)} 
			\min\set*{\p*{\frac{\eta \sigma_e}{\tilde D}}^2,\ \p*{\frac{\eta \sigma_e}{\tilde D}}^4}, \\
		      &\frac{n^{2-\kappa}}{\log(p)}  \p*{\frac{\sigma}{\sigma_e}}^2 \p*{\frac{\eta \sigma_e}{\tilde D}}^4,\  
		       \frac{\sigma_e^2}{\log(p) \norm{\psi_{col}'\Sigmacol^{1/2}}^2}\bigg]. 
\end{align*}
\end{enumerate}
Here $\tilde D =  D + \norm{\varepsilon_{ej} - \psi_{col}'\varepsilon_{\cdot} j}_{L_2}$ for $D$ as in Proposition~\ref{proposition:sc-bias-bound}.

In place of the bound on $\rank(A)$, we may substitute the singular value bound 
\[ \sigma_{R+1}(A) \le \sigma p_{eff}^{-1/2} R\ /\ \sqrt{\log(p)} \ \text{ for some }\ 
    R \ll  \min\set*{ n \p*{\frac{\eta \sigma_e}{\tilde D^2}}^2,\ \frac{\sigma^2 n^2}{p_{eff}} \p*{\frac{\eta \sigma_e}{\tilde D}}^4 }. \]
\end{corollary}
Here and throughout, convergence refers to limits along arbitrary sequences of $X_k$,  $y_k$, $\eta_k$, etc. 
and we write $a \ll b$ and $a \lesssim b$ meaning $a=o(b)$ and $a=O(b)$.

If the assumptions above hold, the oracle's bias $\E \tilde\tau - \tau$ is negligible,
and we have a consistent variance estimator $\hat\sigma_{\tau}^2$ in the sense that
$\hat \sigma_{\tau}^2 / \sigma_{\tau}^2$ converges to one in probability, then we can construct confidence intervals 
$\hat C_{\alpha} = \hat \tau \pm z_{\alpha/2} \ \hat \sigma_\tau$ with asymptotically correct coverage \smash{$\lim P(\tau \in \hat C_{\alpha}) = 1-\alpha$.}
%

\subsubsection{Understanding the Assumptions}
\label{sec:deviation-conditions}

When autocorrelation is sufficiently small ($\norm{\psi_{col}'\Sigmacol^{1/2}} \approx 0$), 
the summary $D$ of how `typical' the vector of post-treatment observations $a_e$ is of pre-treatment observations  (rows of $A$)
is the key determinant of our bound on $\hat \tau - \tilde \tau$ \eqref{eq:sc-deviation}.
It therefore determines the conditions above on $p_{eff}$ and $\rank(A)$ for asymptotic normality.
Throughout we will assume that the first numbered condition of Corollary~\ref{corollary:sc-normality} holds for $\kappa=1$.
This does not limit us much, as it unlikely that the oracle estimator's bias $\E \tilde\tau - \tau = b_e - a_e'\stheta$ 
will be $o(\sigma_{\tau})$ if the root-mean-squared element of the `training error' $A\stheta - b$ is not.

The best possible case is that the post-treatment observations aligns perfectly with the dominant 
direction in $A$: $a_e = \norm{a_e} v_1$ where $v_1$ is the first right singular vector of $A$. In this case,
$D \approx \sigma \eta \sqrt{n} \norm{a_e} / \sigma_1$. This should be comparable to $\sigma \eta$
because the length $\norm{a_e}$ of the vector $a_e$ will scale like the square root of its dimension $p$ 
and the largest singular values of an $n \times p$ matrix $A$ will tend to be on the order $\sqrt{np}$ --- 
this is the case, at least, when $A$ is a rank-1 matrix $A=xy'$, as then $\sigma_1 = \norm{x}\norm{y}$. 
More generally, it should be comparable to $\sigma \eta \sqrt{R_e}$ if $a_e$ or a good approximation to it 
is spanned by the first $R_e$ singular vectors and its projection on the $k$th
decays like that of a typical row of $A$ in the sense that
$\abs{a_e' v_k}/\abs{a_e' v_1} \lesssim \sigma_k/\sigma_1 = \norm{A v_k}/\norm{A v_1}$.\footnote{An approximation $a_e^R$
is good enough if it satisfies $\norm{a_e-a_e^R} \lesssim \sigma \eta \sqrt{R}$. In particular,
taking $a_e^R$ to be the projection of $a_e$ on the first $R_e$ singular vectors,
it suffices that the squared length of the orthogonal component, $\norm{a_e-a_e^R}^2=\sum_{k > R} (a_e' v_k)^2$,
satisfies $\sum_{k > R} (a_e' v_k)^2 \lesssim \sigma^2 \eta^2 R$.}

When there is enough typicality and little enough autocorrelation, this parameter $R_e$
determines our conditions on $p_{eff}$ and \smash{$\rank(A)$}. In particular, if there is nontrivial noise both pre and post-treatment 
($\sigma_e \gtrsim \sigma \gtrsim 1$), then it suffices that $p_{eff} \ll n/\{R_e^2 \log(p)\}$ and $\rank(A) \ll n/R_e$.
In the panel data literature, asymptotics involving a factor model $A=\sum_{k \le R_0} \sigma_k u_k v_k$ 
of bounded rank $R_0$ are common \citep[e.g.,][]{abadie2010synthetic, bai2009panel}.
As $R_e \le \rank(A)$, in this regime it suffices that $n \to \infty$ and $p_{eff} \ll n/\log(p)$.
In particular, this essentially holds when $y$ is an average of $p_e$ treated units like the control units as in \eqref{eq:synthdid-disaggregated}
and the number of treated units is negligibly small relative to the number of pre-treatment periods.

These ideal-case results do not depend on the strength of regularization $\eta$ except in that $\eta$ affects
the oracle estimator's bias. Qualitatively, if $a_e$ is sufficiently typical,
the behavior of $a_e'\hat \theta$ is essentially determined by that of $A \hat \theta$ and the squared loss $\norm{A\theta - b}^2$
determines this well enough that regularization is irrelevant. When it is less typical, our results become sensitive to $\eta$,
as regularization plays a larger role in driving the convergence of $\hat \theta$. 
If $a_e$ is atypical enough, e.g., orthogonal to the rows of $A$, then $D \approx \norm{a_e}$ and will generally be on the order of $\sqrt{p}$,
and Corollary~\ref{corollary:sc-normality} requires that \smash{$\eta \gg \{ p_{eff} p^2 \log(p) / n \}^{1/4}$} when $\sigma_e$ is bounded.
Unless $p$ and $p_{eff}$ are relatively small, this is fairly strong regularization, and in some cases it is enough to essentially rule out 
the possibility that the oracle is negligibly biased. Thus, if we do not expect $a_e$ to be very typical,
there are stricter limits on the study designs ($p_e,p,n$) in which we should expect approximate normality. 

\subsection{The Bias of the Oracle Estimator}
\label{sec:oracle-behavior}

The oracle estimator's bias is the error in predicting the expected post-treatment outcome for the treated unit by that of the synthetic control.
\[ \tilde\tau - \tau = y_e-\tau - x_e' \tilde\theta = (b_e' - a_e \tilde\theta) + (\nu_e - \varepsilon_e'\tilde\theta). \]
Intuitively, this will usually be no smaller than the typical error of analogous pre-treatment predictions, i.e.,
the root-mean-squared error $n^{-1/2}\norm{A\stheta - b}$ that is approximately minimized by the oracle weights. 
If we are lucky, the two will be comparable. More generally, we can hope that post-treatment prediction error 
$b_e - a_e \tilde\theta$ will be comparable to the prediction errors in some fraction of the pre-treatment periods.
If $\xi$ is the typical error in its worst $n^{\kappa}$ pre-treatment periods, then \smash{$\xi \le n^{-\kappa/2}\norm{A\stheta - b}$} \smash{$= n^{(1-\kappa)/2} \cdot n^{-1/2}\norm{A\stheta - b}$},
and it follows that this typical error is negligible if the pre-treatment error 
satisfies the bound \smash{$n^{-1/2}\norm{A\stheta - b} \ll n^{(\kappa-1)/2} \sigma_{\tau}$} 
as in Corollary~\ref{corollary:sc-normality}.

The oracle weights will fit this well if, roughly speaking, there exist weights on `enough of the controls' that do. 
Consider weights $\theta \in \Theta$ with \smash{$p^{-\kappa'/2}\norm{\theta} \lesssim p^{-\kappa'}$},
e.g., those distributing weight roughly evenly over $p^{\kappa'}$ controls units.
As the optimality of $\tilde\theta$ in oracle least squares problem \eqref{eq:ridge-oracle}
implies that oracle mean squared error is bounded by \smash{$n^{-1}\norm{A\theta - b}^2 + (\eta \sigma)^2\norm{\theta}^2$} for any $\theta \in \Theta$,
oracle mean squared error will be smaller than \smash{$n^{\kappa-1} \sigma_{\tau}^2$} (as desired)
if such weights for \smash{$p^{\kappa'} \gg \eta^2 \sigma^2 n^{1-\kappa} / \sigma_{\tau}^2$}
have error like this. Expanding \smash{$\sigma_{\tau}^2 = \sigma_e^2 / p_{eff}$}, this means it suffices to get good fit
from roughly even weights on \smash{$p^{\kappa'} \gg \eta^2  n^{1-\kappa} p_{eff}$} units when 
there is nontrivial post-treatment noise \smash{($\sigma_e \gtrsim \sigma$).} This is equivalent to a restriction
\smash{$\eta \ll \{n^{\kappa-1} p^{\kappa'}/p_{eff}\}^{1/2}$} on the strength of regularization 
where \smash{$n^{\kappa}$} and \smash{$p^{\kappa'}$} are the numbers of `predictive' pre-treatment periods and control units.

In the ideal case discussed in Section~\ref{sec:deviation-conditions}, essentially any level of regularization is
sufficient to ensure the estimator's normality around the oracle's mean. Thus, we can for example estimate weights by least squares
without penality $(\eta = 1)$, which will satisfy this negligible bias condition if 
$p_{eff} \ll n^{\kappa-1}p^{\kappa'}$.  For example, if $\sqrt{n}$ of the pre-treatment observations and $\sqrt{p}$
of the controls are predictive in the senses above, this imposes the condition $p_{eff} \ll \sqrt{np}$.
Considering also the ideal-case condition $p_{eff} \ll n/\log(p)$ for normality,
these assumptions imply that when we have far fewer treated observations than both control observations and pre-treatment periods,
we should expect $\hat\tau \pm 1.96\sigma_{\tau}$ to be an approximately valid $95\%$ confidence interval for $\tau$.

In the far-from-ideal case discussed at the end of Section~\ref{sec:deviation-conditions}, conclusions are different.
To satisfy our sufficient conditions for negligible bias and normality, we must regularize with strength $\eta$ satisfying
$\{p_{eff} p^2 \log(p) / n \}^{1/4}  \ll \eta \ll  \{n^{\kappa-1} p^{\kappa'}/p_{eff}\}^{1/2}.$
This leaves valid options only when \smash{$p_{eff}^3 \log(p) \ll n \cdot n^{2(\kappa-1)} p^{2(\kappa'-1)}$}.
If we consider the case that barely more than $\sqrt{n}$ of the pre-treatment observations are predictive ($\kappa=1/2+\epsilon$),
we need all controls to be predictive ($\kappa'=1$) for our results to apply with even one treated unit. 

Behavior in practice may not reflect what we see in this extremely unfavorable case, but it may not reflect what we see in the ideal case either.
To better understand it, it is important to have a sense of where we tend to fall on the continuum between these extremes
and what additional structure (e.g., approximate randomization) exists in the data.
For that, we must consider the way the onset of treatment and assignment of units to treatment is decided.

\bibliographystyle{plainnat}
\bibliography{eiv-references}

\newpage 

\begin{appendix}

\section{Refined Results on Least Squares}
\subsection{Definitions}
To state our refined results, we will use the following definitions.
\begin{equation}
\label{eq:more-common-notation}
\begin{array}{l||l|l}
                       & \text{independent rows} 
		       & \text{independent columns} \\

 \width_{\Sigma}(S)    &  \ge \width(\Sigmarow^{1/2} S)   
		       &  \ge \norm{\Sigmacol}^{1/2}\width(S)  \\
 p_{eff,\Sigma}^{-1/2} & \norm{\Sigmarow^{1/2}(\stheta - \psi)} + \norm{\varepsilon_{i\cdot} \psi - \nu_i}_{L_2} 
		       &  \norm{\Sigmacol}^{1/2}\norm{\stheta} + \norm{\Sigmanu}^{1/2} \\
 K                     & K_{row}
		       & K_{col} \\
 \phi                  &  1  
		       &  \sqrt{\frac{\norm{\Sigmacol}}{n^{-1}\trace(\Sigmacol)}} 
\end{array}
\end{equation}
Here our notation means that $\width_{\Sigma}(S)$ can be any upper bound on 
$\width(\Sigmarow^{1/2} S)$ or $\norm{\Sigmacol^{1/2}}\width(S)$ and 
$K$ is a bound characterizing the concentration of quantities related to $\varepsilon$ and $\nu$.
\begin{equation}
\label{eq:K-definition}
\begin{aligned}
&K_{row} \ge \max\p*{\norm{\varepsilon_{i\cdot}\Sigmarow^{-1/2}}_{\psi_2}, 
				  \frac{\norm{\varepsilon_{i\cdot}\psi - \nu_i}_{\psi_2}}{\norm{\varepsilon_{i\cdot}\psi - \nu_i}_{L_2}}}, \\ 
&K_{col} \ge \max\p*{\norm{\Sigmacol^{-1/2}\varepsilon_{\cdot j}}_{\psi_2}, \norm{\Sigmanu^{-1/2}\nu}_{\psi_2}},\ \text{ and for all }\ u \ge 0, \\
&P\p{\abs{\norm{\varepsilon_{\cdot j}}^2 - \E \norm{\varepsilon_{\cdot j}}^2 } \ge u} \
    \le c\exp\p*{-c \min\p*{\frac{u^2}{K_{col}^4 \E\norm{\varepsilon_{\cdot j}}^2},\  \frac{u}{K_{col}^2\norm{\Sigmacol}}}}.
\end{aligned}
\end{equation}
In this statement, $\norm{\xi}_{\psi_2} := \sup_{\norm{v} \le 1}\norm{v'\xi}_{\psi_2}$ for a subgaussian vector $\xi$. 
In the gaussian case, we can take $K$ to be universal constant, as the subgaussian norm $\norm{z}_{\psi_2}$
of a gaussian random variable is equivalent to its $L_2$ norm $\norm{z}_{L_2}$ and the 
the Hanson-Wright inequality gives the tail bound on $\abs{\norm{\varepsilon_{\cdot j}}^2 - \E \norm{\varepsilon_{\cdot,j}}^2}$.

\subsection{Results}

\begin{theorem}
\label{theorem:rate-general}
Suppose $v,s \in \mathbb{R}$ with $v \ge 1$ and $R \in \mathbb{N}$ satisfy this fixed point condition:
\begin{align*} 
s^2 \ge c\bigg[&\frac{K^4v^2\width^2_{\Sigma}(\sTheta_s)\set{1 + \phi (R/n)^{1/2}}^2}{\eta_R^4 n} + \frac{K^2v^2 \width^2_{\Sigma}(\sTheta_s)}{\eta_R^2 n} \\
    +   &\frac{Kv \norm{A\stheta - b} \width_{\Sigma}(\sTheta_s) + K^2v (n/p_{eff,\Sigma})^{1/2} \width_{\Sigma}(\sTheta_s) 
	  + (K^2v^2 R/p_{eff,\Sigma})^{1,1/2}}{\eta_R^2 n}\bigg], \\
R &\ge c \sigma_{R+1}(A)\width(\sTheta_s)/(\phi s + v p_{eff,\Sigma}^{-1/2}), \\
\eta_R^2 &= \max(0,\ \eta^2 - cK^2\phi^2R/n), \ x^{1,1/2} = x + x^{1/2}. 
\end{align*}
Then in the setting described in Section~\ref{sec:setting}, with either the rows of $[\varepsilon, \nu]$ independent and identically distributed
or the columns of $[\varepsilon, \nu]$ independent with those of $\varepsilon$ identically distributed,
$\norm{\Sigmarow^{1/2} (\htheta - \stheta)} \le  s$ and 
$\norm{A (\htheta-\stheta) + \htheta_0 - \theta_0} \le \eta n^{1/2}s$ 
with probability $1-c\exp\sqb{-\min\set{v^2 \phi^{-4} \width_{\Sigma}^2(\sTheta_s)/s^2,\ v^2 R,\ n}}$.
\end{theorem}
\noindent Theorem~\ref{theorem:rate-simplified} is a simplified variant that ignores some of the subtleties of the noise
by substituting for $K$ and $\phi$ unspecified universal constant bounds
and taking $\width_{\Sigma}(\sTheta_s) = \norm{\Sigma}^{1/2}\width(\sTheta_s)$.

\section{Proof of Refined Results}
We first address the case with no intercept, then generalize our proof to treat the case with intercept in Section~\ref{sec:with-intercept}. 
We begin at \eqref{eq:excess-sketch} from our proof sketch, which justifies working with the following lower bound $\tilde \excess$ on $\excess$.

\begin{equation}
\label{eq:excess-lb-def}
\begin{aligned}
\tilde\excess(\delta) 
&= \norm{X\delta}^2 + n(\eta^2-1)\norm{\Sigmarow^{1/2}\delta}^2 +\\
&-  2\abs{(\varepsilon\stheta - \nu)' A \delta + 2(A\stheta - b)' \varepsilon \delta 
 + [(\varepsilon\stheta - \nu)' \varepsilon - \E (\varepsilon\stheta - \nu)' \varepsilon] \delta}.
\end{aligned}
\end{equation}

Our proof relies on a few high-probability bounds on the terms in \smash{$\tilde\excess(\delta)$}. Choose $s,r \in \mathbb{R}$ and $R \in \mathbb{N}$
and define $\sTheta := \Theta - \stheta$ and \smash{$\sTheta_s := \set{\delta \in \sTheta  : \norm{\Sigmarow^{1/2} \delta} \le s}$}. 
On an event of probability $1-c\exp\sqb{-c\min\set{v^2 \phi^{-4} \width^2_{\Sigma}(\sTheta)/s^2 ,\ v^2 R ,\ n}}$, all $\delta \in \sTheta_s$ satisfy
\begin{equation}
\label{eq:quadratic-signal-term} 
\begin{aligned}
\norm{X\delta}^2 &\ge \norm{A\delta}^2 + n\norm{\Sigmarow^{1/2} \delta}^2 - cK^2v\width_{\Sigma}(\sTheta_s)  n^{1/2} \norm{\Sigmarow^{1/2} \delta} -  1_{col} cK^2\width^2_{\Sigma}(\sTheta_s) \\
&- \max(\norm{A\delta}^2/r^2, 1) cK\set{ \phi R^{1/2}rs + \phi \sigma_{R+1}(A)\width(\sTheta_s)s + v\width_{\Sigma}(\sTheta_s) r}.
\end{aligned}
\end{equation}
\begin{align}
&\smallabs{(\varepsilon\stheta - \nu)'A \delta} \le \max(\norm{A\delta}/r,1) cKvp_{eff,\Sigma}^{-1/2} \set{\sqrt{R}r + \sigma_{R+1}(A)\width(\sTheta_s)}. 
\label{eq:low-rank-term} \\
&\smallabs{(A\stheta - b)' \varepsilon \delta} \le cKv \norm{A\stheta - b} \width_{\Sigma}(\sTheta_s)
\label{eq:oracle-error-term} \\
&\abs{[(\varepsilon\stheta - \nu)'\varepsilon - \E(\varepsilon\stheta - \nu)'\varepsilon]\delta} 
\le cv K^2 (n/p_{eff,\Sigma})^{1/2} \width_{\Sigma}(\sTheta_s).
\label{eq:quadratic-noise-term} 
\end{align}
We prove these bounds below in Section~\ref{sec:proof-of-bounds}. When they hold, for all $\delta \in \sTheta_s$, 
\begin{equation}
\label{eq:excess-loss-lb}
\begin{aligned}
\tilde\excess(\delta)
&\ge  \norm{A\delta}^2 + \eta^2 n \norm{\Sigmarow^{1/2}\delta}^2 - cK^2v \width_{\Sigma}(\sTheta_s)n^{1/2}\norm{\Sigmarow^{1/2}\delta} - c\alpha \\
&- cK\sqb*{\set{q^2(\delta) \phi s +  q(\delta) vp_{eff,\Sigma}^{-1/2}}\set{R^{1/2}r + \sigma_{R+1}(A)\width(\sTheta_s)} + v\width_{\Sigma}(\sTheta_s)r}, \\
q(\delta) &= \max(\norm{A\delta} / r, \ 1), \\
\alpha &= 1_{col} K^2 \width^2_{\Sigma}(\sTheta_s) + Kv \norm{A\stheta - b} \width_{\Sigma}(\sTheta_s) + v K^2 (n/p_{eff,\Sigma})^{1/2} \width_{\Sigma}(\sTheta_s).
\end{aligned}
\end{equation}
Throughout, we will work on an event on which this bound holds for one choice of $s$, determined in Section~\ref{sec:s-bound},
and two choices of $r$. The first choice of $r$ is used in Section~\ref{sec:s-bound} and the second, which will be larger, is used in Section~\ref{sec:r-bound}.

To simplify our notation, we will restrict our choice of $R$ (as a function of $s,r$) to those satisfying $R^{1/2}r \ge \sigma_{R+1}(A)\width(\sTheta_s)$,
allowing us to ignore factors of $\sigma_{R+1}(A)$. As the value of $r$ is used in Section~\ref{sec:s-bound} is the smaller of the two,
we derive the sufficient condition in our theorem statement there.

\subsection{Bounding $\norm{\Sigma^{1/2}\delta}$} 
\label{sec:s-bound}
In this section, we will show that $\tilde\excess(\delta) > 0$ for $\delta \in \Theta^\circ_s := \set{\delta \in \sTheta : \norm{\Sigmarow^{1/2}\delta} = s}$.
This extends to $\delta \in \sTheta$ with $\norm{\Sigma^{1/2}\delta} \ge s$ by a simple scaling argument:
because $\sTheta$ is convex and contains zero, any such $\delta$ can be written as $\alpha \delta^\circ$ for $\alpha \ge 1$ and $\delta^\circ \in \Theta^\circ$,
and $\tilde\excess(\alpha \delta^\circ) \ge \alpha^2 \tilde\excess(\delta^\circ)$ for $\alpha \ge 1$.
To show that  $\tilde\excess(\delta) > 0$ for $\delta \in \Theta^\circ_s$, consider the implications of our lower bound \eqref{eq:excess-loss-lb} for such $\delta$.
\begin{equation}
\label{eq:excess-s}
\begin{aligned}
\tilde\excess(\delta)
&\ge  \eta^2 n s^2 - cK^2v \width_{\Sigma}(\sTheta_s)n^{1/2}s - c\alpha \\
&+ \norm{A\delta}^2 - cK \set{q^2(\delta) \phi s R^{1/2} + q(\delta) vp_{eff,\Sigma}^{-1/2}R^{1/2} + v\width_{\Sigma}(\sTheta_s)}r.
\end{aligned}
\end{equation}
We minimize the second line over $x=\norm{A\delta}$ to get a lower bound. As a function of $x$, this second line is
\begin{align*}
q_1(x) &= x^2 - cK \set{\phi s R^{1/2} +  vp_{eff,\Sigma}^{-1/2}R^{1/2} + v\width_{\Sigma}(\sTheta_s)}r && \text{ if } x \le r, \\
q_2(x) &= x^2(1- c K \phi s R^{1/2}/r) - x cK v p_{eff,\Sigma}^{-1/2}R^{1/2} - cK v\width_{\Sigma}(\sTheta_s)r && \text{ if } x \ge r.
\end{align*} 
Because $q_1(x)$ is increasing, its minimum on its domain is $q_1(0)$. 
Furthermore, so long as the leading coefficient of $q_2$ is positive, it has a unique global minimum,
and if that minimum occurs at $x \le r$, its minimum on its domain is $q_2(r)$.
And because $q_2(r) = r^2 + q_1(0) > q_1(0)$, when the global minimum of $q_2$ occurs at $x \le r$,
we can lower bound the second line in \eqref{eq:excess-s} by $q_1(0)$. We choose $r$ to make this happen:
as the minimum of the polynomial $a_2 x^2 - a_1 x - a_0$ occurs at $x= a_1 / 2a_2$,
it requires that $cK v p_{eff,\Sigma}^{-1/2}R^{1/2}/ 2(1- c K \phi s R^{1/2}/r) \le r$,
which holds for $r = cK (\phi s + v p_{eff,\Sigma}^{-1/2})R^{1/2}$. 
This choice makes the leading coefficient of $q_2$ positive as required.

Substituting this lower bound $q_1(0)$ for the last line of \eqref{eq:excess-s}
and expanding our simplifying assumption $R^{1/2}r \ge \sigma_{R+1}(A)\width(\sTheta_s)$
using this choice of $r$, we get the following bound. For all $\delta \in \Theta^\circ_s$,
\begin{equation*}
\label{eq:excess-s-lb}
\begin{aligned}
\tilde\excess(\delta)
&\ge  \eta^2 n s^2 - cK^2v \width_{\Sigma}(\sTheta_s)n^{1/2}s - c\alpha \\
&- cK^2 (\phi s +  vp_{eff,\Sigma}^{-1/2})^2 R - cK^2 v\width_{\Sigma}(\sTheta_s) (\phi s +  vp_{eff,\Sigma}^{-1/2}) R^{1/2} \\ 
&\ge  (\eta^2 n - cK^2 \phi^2 R)s^2 - cK^2v\width_{\Sigma}(\sTheta_s)(n^{1/2} + \phi R^{1/2})s - c\alpha', \\
\alpha' &= \alpha + v^2 K^2 \set{ p_{eff,\Sigma}^{-1}R +  \width_{\Sigma}(\sTheta_s) p_{eff,\Sigma}^{-1/2} R^{1/2} }, \\
R &\ge c \sigma_{R+1}(A)\width(\sTheta_s)/(\phi s + v p_{eff,\Sigma}^{-1/2}).
\end{aligned}
\end{equation*}
To derive the latter lower bound on $\tilde\excess$, 
we use the bound $(x+y)^2 \le 2x^2 + 2y^2$ on the square of  $\phi s + vp_{eff,\Sigma}^{-1/2}$ and
then group terms. So long as $\eta^2 n - cK^2 \phi^2 R > 0$, this lower bound exceeds zero
whenever $s$ exceeds its larger root, and therefore when $s^2$ exceeds the square of its larger root.
As the larger root of $a_2x^2 - a_1x - a_0$ is $(a_1 + \sqrt{a_1^2 + 4a_0a_2})/ 2a_2 \le \sqrt{a_1^2/a_2^2 + 4a_0/a_2}$,
this happens when
\begin{equation}
\label{eq:s-condition-penultimate}
s^2 \ge \p*{\frac{cK^2v\width_{\Sigma}(\sTheta_s)(n^{1/2} + \phi R^{1/2})}{\eta^2 n - cK^2 \phi^2 R}}^2 + \frac{c\alpha'}{\eta^2 n - cK^2 \phi^2 R}.
\end{equation}

\subsection{Bounding $\norm{A\delta}$}
\label{sec:r-bound}

In this section, we will show that $\tilde\excess(\delta) > 0$ for $\delta \in \Theta^\star_s$ with $\norm{A\delta} \ge r$.
For such $\delta$,
\begin{equation*}
\label{eq:excess-loss-lb-r}
\begin{aligned}
\tilde\excess(\delta)
&\ge  \eta^2 n \norm{\Sigmarow^{1/2}\delta}^2 - cK^2v \width_{\Sigma}(\sTheta_s)n^{1/2}\norm{\Sigmarow^{1/2}\delta}  \\
&+( 1 - cK \phi s R^{1/2}/r)\norm{A\delta}^2  - cK vp_{eff,\Sigma}^{-1/2} R^{1/2} \norm{A\delta} - \set{cKv\width_{\Sigma}(\sTheta_s)r + c\alpha}.
\end{aligned}
\end{equation*}
This is a quadratic in $x=\norm{\Sigma^{1/2}\delta}$ of the form $\eta^2 n x^2 - 2 a_1 x + a_0$, 
and is minimized at $x=a_1/(\eta^2 n)$, where it takes on the value $a_0-a_1^2/(\eta^2 n)$. Our lower bound will be positive if this quantity
exceeds zero. This quantity is quadratic in $x=\norm{A\delta}$. 
\begin{align*}
a_0 - a_1^2/(\eta^2 n) 
&= ( 1 - cK \phi s R^{1/2}/r) \norm{A\delta}^2  - cK vp_{eff,\Sigma}^{-1/2} R^{1/2} \norm{A\delta} \\
&- \set{cKv\width_{\Sigma}(\sTheta_s)r + c\alpha + cK^4v^2 \width^2_{\Sigma}(\sTheta_s)/\eta^2}.
\end{align*}
So long as $1 - cK \phi s R^{1/2}/r > 0$, it will be positive when $\norm{A\delta}$ exceeds its larger root. 
It follows that it will be positive for all $\delta \in \sTheta_s$ with $\norm{A\delta} \ge r$ if
$r$ exceeds its larger root.
And as the larger root of $a_2 x^2 - a_1 x - a_0$ is $(a_1 + \sqrt{a_1^2 + 4a_0a_2})/ 2a_2 < \sqrt{a_1^2/a_2^2 + 4a_0/a_2}$,
a sufficient condition is that
\[ r^2 \ge c\p*{\frac{K vp_{eff,\Sigma}^{-1/2} R^{1/2}}{1 - cK \phi s R^{1/2}/r}}^2 
      +    c\frac{Kv\width_{\Sigma}(\sTheta_s)r  + \alpha + K^4v^2 \width^2_{\Sigma}(\sTheta_s)/\eta^2}{1 - cK \phi s R^{1/2}/r}. \]
Multiplying both sides by $(1 - cK \phi s R^{1/2}/r)^2$, we get the equivalent condition,
\begin{align*} 
r^2 &- cK\phi R^{1/2}sr + cK^2 \phi^2 s^2 R \ge cK vp_{eff,\Sigma}^{-1/2} R^{1/2} \\
    &- c(1 - cK \phi s R^{1/2}/r)\set{Kv\width_{\Sigma}(\sTheta_s)r  + \alpha + K^4v^2 \width^2_{\Sigma}(\sTheta_s)/\eta^2}.
\end{align*}
Dropping the $s^2$ term and the factor $(1 - cK \phi s R^{1/2}/r)$ in the last term above yields a simplified sufficient condition:
the nonnegativity of a quadratic function of $r$.
\begin{align*} 
0 \le r^2 &- c\set{K\phi R^{1/2}s + Kv\width_{\Sigma}(\sTheta_s)}r 
	  - c\set{K vp_{eff,\Sigma}^{-1/2} R^{1/2} + \alpha +  K^4v^2 \width^2_{\Sigma}(\sTheta_s)/\eta^2}. 
\end{align*}
This is satisfied when $r$ is at least its larger root, and therefore when
\begin{equation} 
\label{eq:r-condition-penultimate}
r^2 \ge c\set{K\phi R^{1/2}s + Kv\width_{\Sigma}(\sTheta_s)}^2 + c\set{K vp_{eff,\Sigma}^{-1/2} R^{1/2} + \alpha + K^4v^2 \width^2_{\Sigma}(\sTheta_s)/\eta^2}. 
\end{equation}

\subsection{Simplifications}
To simplify our result, we derive a sufficient condition for \eqref{eq:s-condition-penultimate} and \eqref{eq:r-condition-penultimate} to hold.

If $s$ satisfies a variant of \eqref{eq:s-condition-penultimate} in which a slightly different parameter $\alpha''$ replaces $\alpha'$,
then $r=\eta n^{1/2} s$ satisfies \eqref{eq:r-condition-penultimate}. In particular, \eqref{eq:r-condition-penultimate} is satisfied for this $r$ if
\begin{align*} 
(\eta^2 n - cK^2\phi^2 R) s^2 &\ge c\set{ v^2 K^2 \width^2_{\Sigma}(\sTheta_s) +  K vp_{eff,\Sigma}^{-1/2} R^{1/2} + \alpha + K^4v^2 \width^2_{\Sigma}(\sTheta_s)/\eta^2 } \\
			      &= cK^4v^2 \width^2_{\Sigma}(\sTheta_s)/\eta^2 + c \alpha'', \\
\alpha'' &=  \alpha + K^2v^2 \width^2_{\Sigma}(\sTheta_s) + K vp_{eff,\Sigma}^{-1/2} R^{1/2}.
\end{align*}
This is implied by our variant of \eqref{eq:s-condition-penultimate} because 
\[ \frac{ K^4v^2 \width^2_{\Sigma}(\sTheta_s)/\eta^2}{ \eta^2 n - cK^2\phi^2 R } \le \p*{\frac{K^2v\width_{\Sigma}(\sTheta_s)(n^{1/2} + \phi R^{1/2})}{\eta^2 n - cK^2 \phi^2 R}}^2. \]
To see that this holds, cancel like factors and rearrange to get the equivalent condition 
$(1/\eta^2)(\eta^2 n - cK^2 \phi^2 R) \le (n^{1/2} + \phi R^{1/2})^2$, which is obviously satisfied.

It follows that if $s$ satisfies a variant of \eqref{eq:s-condition-penultimate} in which an upper bound on $\max(\alpha',\alpha'')$ replaces $\alpha'$,
then $s$ satisfies \eqref{eq:s-condition-penultimate} and $r=\eta n^{1/2}s$ satisfies \eqref{eq:r-condition-penultimate}. This,
as discussed in Sections~\ref{sec:s-bound} and \ref{sec:r-bound}, implies that 
\smash{$\norm{\Sigmarow^{1/2}\delta} \le s$} and \smash{$\norm{A\delta} \le \eta n^{1/2} s$}
so long as $\eta^2 n - cK^2 \phi^2 R > 0$. We conclude by simplifying this variant.

One upper bound on $\max(\alpha', \alpha'')$ incorporates the excess from $\alpha'$ and $\alpha''$.
\begin{align*} 
\max(\alpha', \alpha'') &\le \alpha + (\alpha' - \alpha) + (\alpha''-\alpha) \\
&=  (1_{col} + 1)K^2v^2 \width^2_{\Sigma}(\sTheta_s) + Kv \norm{A\stheta - b} \width_{\Sigma}(\sTheta_s) + K^2v (n/p_{eff,\Sigma})^{1/2} \width_{\Sigma}(\sTheta_s) \\
&+  K^2v^2 \set{ p_{eff,\Sigma}^{-1}R +  \width_{\Sigma}(\sTheta_s) p_{eff,\Sigma}^{-1/2} R^{1/2} } + K vp_{eff,\Sigma}^{-1/2} R^{1/2}.
\end{align*}
Doubling constant factors justifies dropping the term $v^2K^2 p_{eff,\Sigma}^{-1/2} R^{1/2} \cdot  \width_{\Sigma}(\sTheta_s)$,
as it is the geometric mean of $K^2v^2 \width^2_{\Sigma}(\sTheta_s)$ and $ K^2v^2 p_{eff,\Sigma}^{-1} R$ and therefore bounded by their sum.
Thus, this is bounded by $2\alpha'''$ where, for $x^{1,1/2} = x+x^{1/2}$,
\begin{align*}
\alpha''' &= K^2v^2 \width^2_{\Sigma}(\sTheta_s) + Kv \norm{A\stheta - b} \width_{\Sigma}(\sTheta_s) + K^2v (n/p_{eff,\Sigma})^{1/2} \width_{\Sigma}(\sTheta_s) \\
	  &+ (K^2v^2 R/p_{eff,\Sigma})^{1,1/2}.
\end{align*}

In the statement of Theorem~\ref{theorem:rate-general},
we replace $\eta^2 n - cK^2 \phi^2 R$ with its positive part $\max(0, \eta^2 n - cK^2 \phi^2 R)$
in the denominators in our variant of \eqref{eq:s-condition-penultimate}. 
This allows us to drop our that assumption $\eta^2 n - cK^2 \phi^2 R > 0$, as no $s$ will satisfy the stated bound unless it holds.
For clarity about scaling with $n$, we write this as $\eta_R^2 n$ where $\eta_R = \max(0, \eta^2 - cK^2\phi^2R/n)$.
In this notation, our sufficient condition on $s$ is

\begin{equation}
\label{eq:s-condition-penultimate-2}
s^2 \ge \p*{\frac{cK^4v^2\width^2_{\Sigma}(\sTheta_s)\set{1 + \phi (R/n)^{1/2}}^2}{\eta_R^4 n}} + \frac{c\alpha'''}{\eta_R^2 n}.
\end{equation}

To characterize $R$, recall from our discussion following Equation~\ref{eq:excess-loss-lb} that in this expression,
$R$ can be any value satisfying \smash{$R^{1/2}r \ge \sigma_{R+1}(A)\width(\sTheta_s)$} for both
the value \smash{$r= cK (\phi s + v p_{eff,\Sigma}^{-1/2})R^{1/2}$} used in Section~\ref{sec:s-bound}
and the value \smash{$r=\eta n^{1/2} s$} used here and in Section~\ref{sec:r-bound}. Up to constant factors, the latter is larger:
we know that it exceeds \smash{$cK\phi R^{1/2}s$} because \eqref{eq:r-condition-penultimate} is satisfied and
that it exceeds \smash{$cKv p_{eff,\Sigma}^{-1/2})R^{1/2}$} because \eqref{eq:s-condition-penultimate} implies 
that \smash{$\eta n^{1/2} s \ge c (\eta/\eta_R) (\alpha''')^{1/2} \ge cK v R^{1/2}/p_{eff,\Sigma}^{1/2}.$}
Thus, adjusting constant factors as necessary, the constraint arising from the former value is sufficient.
It is \smash{$R \ge c \sigma_{R+1}(A)\width(\sTheta_s)/(\phi s + v p_{eff,\Sigma}^{-1/2})$.}

We conclude by proving our bounds \eqref{eq:quadratic-signal-term}-\eqref{eq:quadratic-noise-term}.
\subsection{Notation for Proof of Bounds}
\label{sec:proof-of-bounds}

We will show that our bounds \eqref{eq:quadratic-signal-term}, \eqref{eq:low-rank-term}, \eqref{eq:oracle-error-term}, and \eqref{eq:quadratic-noise-term}
hold for all $\delta \in \sTheta_s$ with $\norm{\Sigmarow^{1/2} \delta} \ge v^{-1} n^{-1/2} \width_{\Sigma}(\sTheta_s)$
on events of probability $1-c\exp(-u)$ where, respectively, $u$ is equal to 
\[\min\set{v^2 \phi^{-2} d_{\Sigma}(\sTheta_s), n},\ v^2 d_{\Sigma}(\sTheta_s), \ v^2 R,\ c\min\set{v^2 d_{\Sigma}(\sTheta_s), n}
\ \text{ for }\ v \ge 1.\] 
Thus, by the union bound, all hold simultaneously on an event of probability at least $1-c\exp(-u)$ for $u$ equal to the minimum of the above values.
This event has probability at least $1-c\exp\sqb{-c\set{(v/\phi)^2 d_{\Sigma}(\sTheta_s),\ v^2R,\ n}}$ as claimed,
as $\phi \ge 1$.

To state our bounds in common notation throughout, we recall a few definitions from \eqref{eq:common-notation} and \eqref{eq:more-common-notation} 
that depend on whether the rows or columns of $[\varepsilon, \nu]$ are independent.
\begin{equation}
\label{eq:recall-common-notation}
\begin{array}{l||l|l}
                       & \text{ind. rows} 
		       & \text{ind. columns} \\
\hline
 \width_{\Sigma}(S)    &  \ge \width(\Sigmarow^{1/2} S)   
		       &  \ge \norm{\Sigmacol^{1/2}}\width(S)  \\
 p_{eff,\Sigma}^{-1/2} &  \norm{\Sigmarow^{1/2}(\stheta - \psi)} +  \norm{\varepsilon_{i\cdot} \psi - \nu_i}_{L_2}  
		       &  \norm{\Sigmacol}^{1/2}\norm{\stheta} + \norm{\Sigmanu}^{1/2}  \\
 \phi                  &  1  
		       &  \sqrt{\frac{\norm{\Sigmacol}}{n^{-1}\trace(\Sigmacol)}} = \sqrt{\frac{\norm{\Sigmacol}}{\norm{\Sigmarow}}} \\
 1_{col}               &  0  
		       &  1.
\end{array}
\end{equation}
Here we state a second equivalent definition of $\phi$.
Equivalence for $\phi$ holds because when columns of $\varepsilon$ are iid, $\Sigmarow = n^{-1}\trace(\Sigmacol)I$. We show this via a straightforward calculation.
\begin{align*}
&\Sigmacol^{ij} = p^{-1}\sum_{k \le p}\E \varepsilon_{ik} \varepsilon_{jk} = \E \varepsilon_{i1}\varepsilon_{j1}
\quad \text{ has } \quad  n^{-1}\trace(\Sigmacol)=n^{-1}\sum_{i \le n} \E\varepsilon_{i1}^2 \quad \text{ and therefore } \\
&\Sigmarow^{ij} = n^{-1} \sum_{k \le n} \E \varepsilon_{ki} \varepsilon_{kj}= n^{-1}\sum_{k \le n} \E \varepsilon_{k1}^2 \cdot I^{ij} = 
    n^{-1}\trace(\Sigmacol) \cdot I^{ij}.  
\end{align*}
In terms of these quantities, we will also define a variant of stable dimension,
\begin{equation}
 d_{\Sigma}(S) = \frac{\width^2_{\Sigma}(S)}{\phi^2 \rad^2(\Sigmarow^{1/2}S)}
	       = \begin{cases}
			\width^2_{\Sigma}(S)/\rad^2(\Sigmarow^{1/2}S)    & \text{ ind. rows} \\
			\width^2_{\Sigma}(S)/(\norm{\Sigmacol}\rad^2(S)) & \text{ ind. cols} .
		 \end{cases}
\end{equation}
The second definition is equivalent because when when the columns of $\varepsilon$ are iid, 
$\rad(\Sigmarow^{1/2}S)=\norm{\Sigmarow}^{1/2}\rad(S) = \phi^{-1}\norm{\Sigmacol}^{1/2}\rad(S)$.

\subsection{Proving \eqref{eq:quadratic-signal-term}}

To lower bound $\norm{X\delta}^2$, we work with the following three term decomposition.
\[ \norm{X\delta}^2 = \norm{A\delta}^2 + \norm{\varepsilon\delta}^2 + 2\delta' A' \varepsilon \delta \]
The first term is deterministic, and we will lower bound the remaining two.

\subsubsection{Lower bounding the second term.}
\label{sec:second-term}
\paragraph{With independent rows.} We use a bound from \citet[Theorem 1.4]{liaw2017simple}.
For a random matrix $\xi$ with iid rows with unit covariance, 
\[ \sup_{x \in T}\abs{\norm{\xi x} - \sqrt{n}\norm{x}} \le c \norm{\xi_{i\cdot}}_{\psi_2}^2 (\width(T) + u \rad(T))\ \text{ with probability } 1-\exp(-u^2). \]
Letting $\xi = \varepsilon \Sigmarow^{-1/2}$ and $x=\Sigmarow^{1/2}\delta$, this gives
\begin{align*}
\sup_{\delta \in \sTheta_s} \abs{\norm{\varepsilon\delta} - n^{1/2}\norm{\Sigmarow^{1/2}\delta}} 
&\le c K^2 [\width(\Sigma^{1/2}\sTheta_s) + u\rad(\Sigma^{1/2}\sTheta_s)] \\
&\le c K^2 \width_{\Sigma}(\sTheta_s)[1 + ud^{-1/2}_{\Sigma}(\sTheta_s)].
\end{align*}

Observing that for any nonnegative $x,y,z$ satisfying $\abs{x-y} \le z$,
\[ x^2 - y^2 = (x-y)(x+y) = (x-y)(x-y+2y) = (x-y)^2 + 2y(x-y)  \ge -2yz, \]
it follows that with probability $1-\exp(-u^2)$, for all $\delta \in \sTheta_s$, 
\[ \norm{\varepsilon\delta}^2 \ge  n\norm{\Sigmarow^{1/2}\delta}^2 - c K^2 \width_{\Sigma}(\sTheta_s)[1+ ud^{-1/2}_{\Sigma}(\sTheta_s)] n^{1/2}\norm{\Sigmarow^{1/2}\delta}. \]
Taking $v=ud^{-1/2}_{\Sigma}(\sTheta_s)$, it follows that with probability $1-\exp(-v^2d_{\Sigma}(\sTheta_s))$,
\[ \norm{\varepsilon\delta}^2 \ge  n\norm{\Sigmarow^{1/2}\delta}^2 - c K^2 (1+v)\width_{\Sigma}(\sTheta_s) n^{1/2}\norm{\Sigmarow^{1/2}\delta}. \]

\paragraph{With independent columns.} We use Corollary~\ref{cor:norm-of-sum-bound},
taking $\xi=\varepsilon'$ and $\mathcal A = \sTheta_s$. It implies that with probability $1-2\exp(-u^2)$, 
for all $\delta \in \sTheta_s$,
\begin{align*} 
&\norm{\varepsilon\delta}^2 \ge \E \norm{\varepsilon \delta}^2 \\
&- cK^2\set*{ \norm{\Sigmacol}^{1/2}\trace^{1/2}(\Sigmacol)\width(\sTheta_s)\norm{\delta} +
	\norm{\Sigmacol}\width^2(\sTheta_s) + \norm{\Sigmacol}\norm{\delta} \rad(\sTheta_s)(n^{1/2}u + u^2)}. 
\end{align*}						    
Recall that $\E \norm{\varepsilon \delta}^2 = n\norm{\Sigmarow^{1/2} \delta}^2$ and $\Sigmarow=n^{-1}\trace(\Sigmacol)I$, so
\begin{align*} 
\trace^{1/2}(\Sigmacol)\norm{\delta} &= n^{1/2}\norm{\Sigmarow^{1/2}\delta}, \\
\norm{\Sigmacol}\norm{\delta} &= \norm{\Sigmacol}^{1/2} \phi\norm{\Sigmarow}^{1/2} \norm{\delta} = \phi \norm{\Sigmacol}^{1/2} \norm{\Sigmarow^{1/2}\delta}
\end{align*}
and equivalently,
\begin{align*}
\norm{\varepsilon\delta}^2 
&\ge n\norm{\Sigmarow^{1/2} \delta}^2 - cK^2 \norm{\Sigmacol}^{1/2}\width(\sTheta_s)  n^{1/2} \norm{\Sigmarow^{1/2} \delta}  \\
&- cK^2\set{ \norm{\Sigmacol}\width^2(\sTheta_s) + \phi\norm{\Sigmacol}^{1/2}\norm{\Sigmarow^{1/2}\delta}\rad(\sTheta_s)(n^{1/2}u + u^2)}.
\end{align*}
The substitution of $\width_{\Sigma}(\sTheta_s) \ge \norm{\Sigmacol}^{1/2}\width(\sTheta_s)$ yields a looser bound,
\[ n\norm{\Sigmarow^{1/2} \delta}^2 - cK^2 n^{1/2} \norm{\Sigmarow^{1/2} \delta} \set*{\width_{\Sigma}(\sTheta_s) 
+ \phi\norm{\Sigmacol}^{1/2}\rad(\sTheta_s)(u + n^{-1/2}u^2)}  - cK^2 \width_{\Sigma}^2(\sTheta_s). \]
Take $u = \min\set{n^{1/2},\ v\width_{\Sigma}(\sTheta_s)/\phi\norm{\Sigmacol}^{1/2}\rad(\sTheta_s)}$
and recall that $d_{\Sigma}^{1/2}(\sTheta) = \width_{\Sigma}(\sTheta_s) / $ $ \norm{\Sigmacol}^{1/2}\rad(\sTheta_s)$.
It follows that with probability $1-2\exp\sqb{-\min\set{v^2 \phi^{-2} d_{\Sigma}(\sTheta_s), n}}$, for all $\delta \in \sTheta_s$,
\[ \norm{\varepsilon\delta}^2 \ge n\norm{\Sigmarow^{1/2} \delta}^2 - cK^2(1+v) \width_{\Sigma}(\sTheta_s)  n^{1/2} \norm{\Sigmarow^{1/2} \delta} - 
    cK^2 \width_{\Sigma}^2(\sTheta_s).  \]

\subsubsection{Bounding the third term.}
\label{sec:third-term}
\paragraph{With independent rows.}
Chevet's inequality (Lemma~\ref{lemm:chevet}) implies that when $\varepsilon$ has iid rows,
with probability $1-c\exp(-u^2)$, all $\delta \in \sTheta_{s,r} := \set{\delta \in \sTheta_s : \norm{A\delta} \le r}$
satisify
\begin{align*} 
\abs{ \delta' A' \varepsilon \delta} 
&\le cK\set{ \width(A\sTheta_{s,r}) \rad(\Sigmarow^{1/2}\sTheta_s) + \width(\Sigmarow^{1/2}\sTheta_{s})r
+ u\rad(\Sigmarow^{1/2}\sTheta_s)r }.
\end{align*}
In the notation of Lemma~\ref{lemm:chevet}, this follows by taking 
$V=A\sTheta_{s,r}$, $\xi = \varepsilon \Sigmarow^{-1/2}$, $W=\Sigmarow^{1/2}\sTheta_{s}$.
And taking $v=u\rad(\Sigmarow^{1/2}\sTheta_s)/\width_{\Sigma}(\sTheta_{s})=u d^{-1/2}_{\Sigma}(\sTheta_s)$
and recalling that $\phi=1$ in this case, this implies that with probability $1-c\exp(-v^2 d_{\Sigma}(\sTheta_s))$,
all $\delta \in \sTheta_{s,r}$ satisfy
\begin{align*} 
\abs{ \delta' A' \varepsilon \delta} &\le cK\set{\width(A\sTheta_{s, r}) \phi s + (1+v)\width_{\Sigma}(\sTheta_{s}) r }. 
\end{align*}

\paragraph{With independent columns.}
Similarly, when $\varepsilon$ has iid columns, with probability $1-c\exp(-u^2)$, all $\delta \in \sTheta_{s,r}$ satisfy
\begin{align*}
\abs{ \delta' A' \varepsilon \delta} 
&\le cK\set{ \width(\Sigmacol^{1/2}A\sTheta_{s,r}) \rad(\sTheta_{s}) 
	    + \width(\sTheta_{s})\rad(\Sigmacol^{1/2}A\sTheta_{s,r})  + u \rad(\sTheta_{s}) \rad(\Sigmacol^{1/2}A\sTheta_{s,r})} \\
&\le cK \norm{\Sigmacol}^{1/2} \set{ \width(A\sTheta_{s,r}) \rad(\sTheta_{s}) + \width(\sTheta_{s})r 
	+ u \rad(\sTheta_{s}) r } \\
&\le cK [ 
\width(A\sTheta_{s,r})\phi \rad(\Sigmarow^{1/2}\sTheta_{s}) + \width_{\Sigma}(\sTheta_{s})r 
	+ u \norm{\Sigmacol}^{1/2}\rad(\sTheta_{s}) r.  ] 
\end{align*}
In the notation of Lemma~\ref{lemm:chevet}, this follows by taking 
$V=\sTheta_{s}$,  $\xi = \varepsilon' \Sigmacol^{-1/2}$,  $W=\Sigmacol^{1/2} A\sTheta_{s,r}$.
And taking $v=u \norm{\Sigmacol}^{1/2}\rad(\sTheta_{s}) / \width_{\Sigma}(\sTheta_{s}) = ud^{-1/2}_{\Sigma}(\sTheta_s)$, this implies that with probability $1-c\exp\set{v^2 d_{\Sigma}(\sTheta_s)}$, all $\delta \in \sTheta_{s,r}$ satisfy
\begin{align*} 
\abs{ \delta' A' \varepsilon \delta} 
&\le  cK \set{ \width(A\sTheta_{s,r}) \phi s  + (1+v)\width_{\Sigma}(\sTheta_{s})r}.
\end{align*}
This matches our bound from the case with independent rows. 

\paragraph{A more concrete bound.}
We make our bound above more concrete by substituting a bound on $\width(A\sTheta_{s,r})$.
To bound $\width(A\sTheta_{s,r})$, decompose $A$ as $A_R + (A-A_R)$ where $A_R$ is a rank-$R$ approximation
to $A$. Because $A\sTheta_{s,r} \subseteq A_R\sTheta_{s,r} + (A-A_R)\sTheta_{s,r}$
and $\width(S+T)=\width(S) + \width(T)$ for any $S,T$ \citep[e.g.,][Proposition 7.5.2]{vershynin2018high}, 
$\width(A\sTheta_{s,r}) \le \width(A_R \sTheta_{s,r}) +  \width\set{(A-A_R)\sTheta_{s,r}}$.
Taking $A_R$ to be the best rank-$R$ approximation in terms of operator norm, the first $R$ 
terms of the singular value decomposition $A=\sum_{k} \sigma_k u_k v_k'$, 
$\norm{A_R x} \le \norm{A x}$ for all vectors $x$. 
Thus, $A_R \sTheta_{s,r}$ is contained in the ball of radius $r$ in the $R$-dimensional image of $A_R$,
which has gaussian width bounded by \smash{$c \sqrt{R}r$} \citep[e.g.,][Example 7.5.7]{vershynin2018high}.
And the gaussian width of $(A-A_R)\sTheta_{s,r}$ is bounded by $\norm{A-A_R}\width(\sTheta_{s,r})$ where $\norm{A-A_R} = \sigma_{R+1}(A)$,
the $(R+1)$st singular value of $A$. Thus, on an event of probability $1-c\exp\set{v^2 d_{\Sigma}(\sTheta_s)}$,
all $\delta \in \sTheta_{s,r}$ satisfy
\[ \abs{\delta' A' \varepsilon \delta} 
\le cK\set{ \phi \sqrt{R}rs + \phi \sigma_{R+1}(A)\width(\sTheta_s)s + (v+1)\width_{\Sigma}(\sTheta_s) r}. \]
Now consider $\delta \in \sTheta_s$ with $\norm{A\delta} > r$ and define $\tilde\delta = (r/\norm{A\delta})\delta \in \sTheta_{s,r}$.
As $\abs{\delta' A' \varepsilon \delta} = (\norm{A\delta}^2/r^2) \abs{\tilde\delta' A' \varepsilon \tilde\delta}$,
the bound above implies that all $\delta \in \sTheta_s$ satisfy
\[ \abs{\delta' A' \varepsilon \delta} \le \max(\norm{A\delta}^2/r^2,1) cK\set{ \phi \sqrt{R}rs + \phi \sigma_{R+1}(A)\width(\sTheta_s)s + (v+1)\width_{\Sigma}(\sTheta_s) r}. \]

\subsubsection{Combining bounds}
By the union bound, the bounds derived in Sections~\ref{sec:second-term} and \ref{sec:third-term} hold with probability at least
$1- c \exp\sqb{-\min\set{v^2\phi^{-2} d_{\Sigma}(\sTheta_s), n, v^2 d_{\Sigma}(\sTheta_s)}}$, and together they imply that all $\delta \in \sTheta_s$ satisfy
\begin{align*} 
\norm{X\delta}^2 &\ge \norm{A\delta}^2 + n\norm{\Sigmarow^{1/2} \delta}^2 - cK^2(v+1) \width_{\Sigma}(\sTheta_s)  n^{1/2} \norm{\Sigmarow^{1/2} \delta} - 1_{col} cK^2\width^2_{\Sigma}(\sTheta_s) \\ 
&- \max(\norm{A\delta}^2/r^2,1) cK\set{ \phi \sqrt{R}r s + \phi \sigma_{R+1}(A)\width(\sTheta_s)s + (v+1)\width_{\Sigma}(\sTheta_s) r}.
\end{align*} 
Imposing the constraint $v \ge 1$ justifies replacing $v+1$ with $v$; this yields \eqref{eq:quadratic-signal-term}.

\subsection{Proving \eqref{eq:low-rank-term}}
\label{sec:low-rank-term}
$(\varepsilon\stheta - \nu)'A\delta = z' S^{1/2} A\delta$ where 
$z=S^{-1/2}(\varepsilon\stheta - \nu)$ is a subgaussian vector.
By Talagrand's majorizing measures theorem \citep[Corollary 8.6.3]{vershynin2018high},
this is bounded by $c \norm{z}_{\psi_2} \set{\width(S^{1/2}A\sTheta_{s,r}) + u\rad(S^{1/2}A\sTheta_{s,r})}$
for all $\delta \in \sTheta_{s,r}$ on an event of probability $1-2\exp(-u^2)$.

To bound $\norm{z}_{\psi_2}$, we will use the following consequence of Hoeffding's inequality \citep[e.g.,][Theorem 2.6.3]{vershynin2018high}.
For any vector $x$ and random matrix $\xi$, if $\xi$ has independent rows or columns respectively,
\begin{equation}
\label{eq:subgaussian-norm-mv}
\begin{aligned}
\norm{\xi x}_{\psi_2}^2 := \sup_{\norm{y}\le 1}\norm{y'\xi x}_{\psi_2}^2 
&\le \sup_{\norm{y} \le 1} c\sum_i y_i^2\norm{\xi_{i \cdot} x}_{\psi_2}^2 
\le c\norm{x}^2 \max_i \norm{\xi_{i \cdot}}_{\psi_2}^2, \\
&\le \sup_{\norm{y} \le 1} c\sum_j \norm{y'\xi_{\cdot j}}_{\psi_2} x_j^2  
\le  c\norm{x}^2 \max_j \norm{\xi_{\cdot j}}_{\psi_2}^2.
\end{aligned}
\end{equation}

\paragraph{With independent rows.}
When the rows of $[\varepsilon, \nu]$ are independent, we take $S=I$. 
By the triangle inequality, \eqref{eq:subgaussian-norm-mv} and Hoeffding's inequality, and the definition of $K$, 
\begin{align*}
\norm{z}_{\psi_2} 
&\le \norm{\varepsilon (\stheta - \psi)}_{\psi_2} + \norm{\varepsilon \psi - \nu}_{\psi_2} \\
&\le c\max_i \norm{\varepsilon_{i \cdot}\Sigmarow^{-1/2}}_{\psi_2}\norm{\Sigmarow^{1/2}(\stheta - \psi)} +  c\max_i \norm{\varepsilon_{i\cdot} \psi - \nu_i}_{\psi_2} \\
&\le c K \set{\norm{\Sigmarow^{1/2}(\stheta - \psi)} +  \norm{\varepsilon_{i\cdot} \psi - \nu_i}_{L_2}}.
\end{align*}
Thus,  on an event of probability $1-2\exp(-u^2)$, for all $\delta \in \sTheta_{s,r}$,
\begin{align*}
\abs{(\varepsilon \stheta-\nu)' A \delta} &\le cK 
   \set{\norm{\Sigmarow^{1/2}(\stheta - \psi)} +  \norm{\varepsilon_{i\cdot} \psi - \nu_i}_{L_2}}
   \set{\width(A\sTheta_{s,r}) + u\rad(A\sTheta_{s,r})} \\
&\le cKp_{eff,\Sigma}^{-1/2} \set{\sqrt{R}r + \sigma_{R+1}(A)\width(\sTheta_s)+ ur}.
\end{align*}
To derive the latter bound, we recall that the first parenthesized factor in the former is $p_{eff,\Sigma}$ \eqref{eq:recall-common-notation} 
and substitute the bound on $\width(A\sTheta_{s,r})$ derived in Section~\ref{sec:third-term} above.

\paragraph{With independent columns.}
When the columns of $[\varepsilon, \nu]$ are independent, we use a variation of this argument,
beginning with the triangle inequality bound $\abs{(\varepsilon\stheta - \nu)'A\delta} \le \abs{z_1' \Sigmacol^{1/2}A\delta} + \abs{z_2' \Sigmanu^{1/2}A\delta}$
for $z_1 = \Sigmacol^{-1/2}\varepsilon\stheta$ and $z_2 = \Sigmanu^{-1/2}\nu$. 
Here $\norm{z_1}_{\psi_2} \le cK\norm{\stheta}$ and $\norm{z_2}_{\psi_2} \le K$ by \eqref{eq:subgaussian-norm-mv} and the definition of $K$.
By Talagrand's majorizing measures theorem and the union bound, on an event of probabiltity $1-4\exp(-u^2)$, for all $\delta \in \sTheta_{s,r}$, 
\begin{align*} 
\abs{(\varepsilon\stheta - \nu)'A\delta} &\le  
  c \norm{z_1}_{\psi_2} \set{\width(\Sigmacol^{1/2}A\sTheta_{s,r}) + u\rad(\Sigmacol^{1/2}A\sTheta_{s,r})} \\ 
  &+ c \norm{z_2}_{\psi_2} \set{\width(\Sigmanu^{1/2}A\sTheta_{s,r}) + u\rad(\Sigmanu^{1/2}A\sTheta_{s,r})} \\
&\le cK (\norm{\Sigmacol}^{1/2}\norm{\stheta} +  \norm{\Sigmanu}^{1/2})\set{\width(A\sTheta_{s,r}) + u\rad(A\sTheta_{s,r})} \\
&\le cKp_{eff,\Sigma}^{-1/2} \set{\sqrt{R}r + \sigma_{R+1}(A)\width(\sTheta_s)+ ur}.
\end{align*}
The derive the second bound, we use the inequalities $\width(\Sigma^{1/2}T) \le \norm{\Sigma}^{1/2}\width(T)$ 
and $\rad(\Sigma^{1/2}T) \le \norm{\Sigma}^{1/2}\rad(T)$; to derive the third, we recall that the first parenthesized
factor in the second is $p_{eff,\Sigma}$ \eqref{eq:recall-common-notation}.
Modulo differing definitions of $p_{eff,\Sigma}$, this is the same bound as in the case of independent rows.

\paragraph{Generalization outside $\sTheta_{s,r}$.}
Taking $u=v\sqrt{R}$, this common bound implies that on an event of probability $1-4\exp(-v^2 R)$,
for all $\delta \in \sTheta_{s,r}$, 
\[  \abs{(\varepsilon \stheta-\nu)' A \delta} \le c(1+v)Kp_{eff,\Sigma}^{-1/2} \set{v\sqrt{R} r + \sigma_{R+1}(A)\width(\sTheta_s)}. \]
Now consider $\delta \in \sTheta_s$ with $\norm{A\delta} > r$ and define $\tilde\delta = (r/\norm{A\delta})\delta \in \sTheta_{s,r}$.
As  $\abs{(\varepsilon \stheta-\nu)' A \delta}= (\norm{A\delta}/r) \abs{(\varepsilon \stheta-\nu)' A \tilde\delta}$,
on the same event all $\delta \in \sTheta_s$ satisfy
\[ \abs{(\varepsilon \stheta-\nu)' A \delta} \le \max(\norm{A\delta}/r,1) c(1+v)Kp_{eff,\Sigma}^{-1/2} \set{v\sqrt{R} r + \sigma_{R+1}(A)\width(\sTheta_s)}. \]

\subsection{Proving \eqref{eq:oracle-error-term}}
\label{sec:oracle-error-term}
The quantity bounded is $\smallabs{z' S^{1/2} \delta}$ where $S^{1/2} \delta \in S^{1/2}\sTheta_s$ 
and $z=(A\stheta - b)' \varepsilon S^{-1/2}$ is a subgaussian vector. 
By Talagrand's majorizing measures theorem \citep[Corollary 8.6.3]{vershynin2018high}, on an event of probability $1-2\exp(-u^2)$, 
this is bounded for all $\delta \in \sTheta_s$ by $c \norm{z}_{\psi_2} \set{\width(S^{1/2} \sTheta_s) + u \rad(S^{1/2} \sTheta_s)}$.

When $[\varepsilon, \nu]$ has independent rows, we take $S=\Sigmarow$, 
and $\norm{z'}_{\psi_2}=\norm{(\varepsilon \Sigmarow^{-1/2})' (A\stheta - b)}_{\psi_2} \le cK\norm{A\stheta - b}$ by \eqref{eq:subgaussian-norm-mv}.
When it has independent columns, we take $S=I$,
and $\norm{z'}_{\psi_2}=\norm{(\Sigmacol^{-1/2}\varepsilon)'\set{\Sigmacol^{1/2}(A\stheta - b)}}_{\psi_2} \le cK\norm{\Sigmacol^{1/2}(A\stheta - b)}$.
Thus, the following bounds hold for all $\delta \in \sTheta_s$ on an event of probability $1-2\exp(-u^2)$.
\begin{align*}
\smallabs{z' S^{1/2} \delta} 
&\le cK \norm{A\stheta - b} (\width(\Sigmarow^{1/2} \sTheta_s) + u \rad(\Sigmarow^{1/2} \sTheta_s))  && \\
&\le cK \norm{A\stheta - b} (\width_{\Sigma}(\sTheta_s) + u \rad(\Sigmarow^{1/2} \sTheta_s)) && \text{ind. rows} \\
\smallabs{z' S^{1/2} \delta} 
&\le cK \norm{\Sigmacol^{1/2}(A\stheta - b)} (\width(\sTheta_s) + u \rad(\sTheta_s)) && \\
&\le cK \norm{A\stheta - b} \norm{\Sigmacol}^{1/2} (\width(\sTheta_s) + u \rad(\sTheta_s)) && \\
&\le cK \norm{A\stheta - b} (\width_{\Sigma}(\sTheta_s) + u \norm{\Sigmacol}^{1/2}\rad(\sTheta_s)) && \text{ind. columns}.
\end{align*}
Taking $v = u\rad(\Sigmarow^{1/2} \sTheta_s) / \width_{\Sigma}(\sTheta_s)$ in the first case
and    $v = u\norm{\Sigmacol}^{1/2}\rad(\sTheta_s) / \width_{\Sigma}(\sTheta_s)$ in the second, for which we have the common notation
$v= ud_{\Sigma}^{-1/2}(\sTheta_s)$, this implies that on an event of probability $1-2\exp\set{-v^2 d_{\Sigma}(\sTheta_s)}$,
\[ \smallabs{z' S^{1/2} \delta} \le c(1+v)K \norm{A\stheta - b} \width_{\Sigma}(\sTheta_s) \quad \text{ for all } \delta \in \sTheta_s. \]

\subsubsection{Proving \eqref{eq:quadratic-noise-term}}
We will use different arguments in the cases of independent rows and independent columns. 
However, our conclusions are the same.
With probability $1-6\exp[-c \min\set{v^2 d_{\Sigma}(\sTheta_s), n}]$,
\[ \sup_{\delta \in \sTheta_s} \abs*{(\varepsilon\stheta - \nu)' \varepsilon\delta - \E (\varepsilon\stheta - \nu)' \varepsilon\delta} \le
 cv K^2 (n/p_{eff,\Sigma})^{1/2} \width_{\Sigma}(\sTheta_s). \]

\paragraph{With independent rows.}
This is a multiplier process in the sense of \citet[Theorem 4.4]{mendelson2016upper}. Using their notation,
for $\xi_i = \varepsilon_i\stheta - \nu_i$ and $\F_s = \set{ x' \delta : \delta \in \sTheta_s}$,
\[ \sup_{\delta \in \sTheta_s} \abs*{(\varepsilon\stheta - \nu)' \varepsilon\delta - \E (\varepsilon\stheta - \nu)' \varepsilon\delta} =
    \sup_{f \in \F_s} \abs*{\sum_{i=1}^n \xi_i f(\varepsilon_i) - \E \xi_i f(\varepsilon_i)} \]
and it is bounded by $cv n^{1/2}\norm{\xi_i}_{\psi_2}\tilde\Lambda_{s_0,2}(\F_s)$
with probability $1-4\exp\set{-c \min(v^2 2^{s_0}, n)}$ for any nonnegative integer $s_0$ and $v \ge 1$. Here, 
\[ \tilde\Lambda_{s_0,2}(\F_s) 
\le c\set{ \gamma_2(\F_s, \norm{\cdot}_{\psi_2}) +  2^{s_0/2}\sup_{f \in \F_s} \norm{f}_{\psi_2}}
\le cK\set{\width_{\Sigma}(\sTheta_s) + 2^{s_0/2} \rad(\Sigmarow^{1/2} \sTheta_s)}. \]
The first bound follows directly from the definition of $\tilde \Lambda$ in \citet[][see Definition 1.7 and subsequent discussion]{mendelson2016upper}.
The second follows via Talagrand's majorizing measures theorem \citep[e.g.,][Theorem 8.6.1]{vershynin2018high}
as a consequence of the norm bound
\[ \norm{f}_{\psi_2} = \norm{\varepsilon_i\Sigmarow^{-1/2} \Sigmarow^{1/2}\delta}_{\psi_2} \le 
\norm{\varepsilon_i \Sigmarow^{-1/2}}_{\psi_2}\norm{\Sigmarow^{1/2}\delta}_2 \le K \norm{\Sigmarow^{1/2}\delta}_2
 \quad \text{ for } \quad f(\varepsilon_i) = \varepsilon_i \delta.
\]
Choosing the largest $s_0$ for which $2^{s_0/2} \le \width_{\Sigma}(\sTheta_s)/\rad(\Sigmarow^{1/2}\sTheta_s) = d^{1/2}_{\Sigma}(\sTheta_s)$,
our bound simplifies to $cv n^{1/2}\norm{\xi_i}_{\psi_2} K \width_{\Sigma}(\sTheta_s)$
and holds with probability $1-4\exp[-c \min\set{v^2d_{\Sigma}(\sTheta_s), n}]$.
We bound $\norm{\xi_i}_{\psi_2}$ via triangle inequality as follows.
\begin{align*}
\norm{\xi_i}_{\psi_2} 
&= \norm{\varepsilon_{i\cdot}(\stheta  - \psi)}_{\psi_2} + \norm{\varepsilon_{i\cdot}\psi - \nu_i}_{\psi_2} \\
&= \norm{\varepsilon_{i\cdot}\Sigmarow^{-1/2}}_{\psi_2} \norm{\Sigmarow^{1/2}(\stheta  - \psi)}_{\psi_2} + \norm{\varepsilon_{i\cdot}\psi - \nu_i}_{\psi_2} \\
&\le K\norm{\Sigmarow^{1/2}(\stheta  - \psi)} + K\norm{\varepsilon_{i\cdot}\psi - \nu_i}_{L_2} = Kp_{eff,\Sigma}^{-1/2}. 
\end{align*}
Substituting this bound, with probability $1-4\exp[-c  \min\set{v^2d_{\Sigma}(\sTheta_s), n}]$,
\begin{align*}
&\sup_{\delta \in \sTheta_s} \abs*{(\varepsilon\stheta - \nu)' \varepsilon\delta - \E (\varepsilon\stheta - \nu)' \varepsilon\delta} 
\le cv K^2 (n/p_{eff,\Sigma})^{1/2} \width_{\Sigma}(\sTheta_s).
\end{align*}

\paragraph{With independent columns.}
By the triangle inequality, because $\nu$ and $\varepsilon$ are independent and therefore $\E \nu' \varepsilon \delta = 0$,
\begin{align*} 
\sup_{\delta \in \sTheta_s} \abs{(\varepsilon\stheta - \nu)' \varepsilon\delta - \E (\varepsilon\stheta - \nu)' \varepsilon\delta} 
&\le \sup_{\delta \in \sTheta_s} \abs{\stheta' \varepsilon' \varepsilon \delta - \E \stheta ' \varepsilon' \varepsilon\delta} 
+   \sup_{\delta \in \sTheta_s} \abs{\nu' \varepsilon\delta}.
\end{align*}

The first term here is the supremum of a quadratic subgaussian chaos process in the sense of Lemma~\ref{lemm:weighted-inner-product-bound}. 
Taking $\xi = \varepsilon'$, $v=\stheta$, $W=\Theta_s$ and $\Sigma=\Sigmacol$ in that lemma, with probability $1-2\exp(-u^2)$ for $u \ge 1$, 
\[ \sup_{\delta \in \sTheta_s}\abs{\stheta'(\xi \xi' - \E \xi \xi')\delta }
\le cK^2\norm{\Sigmacol} \norm{\stheta}(n^{1/2} + u)\set{\width(\Theta_s) + u\rad(\sTheta_s)}. \]

The second term here is, conditional on $\nu$, a subgaussian process with increments satisfying
$\norm{\nu' \varepsilon (y-x)}_{\psi_2} \le c\max_j \norm{\nu' \varepsilon_{\cdot j}}_{\psi_2} \norm{y-x}$
by Hoeffding's inequality \citep[e.g.,][Theorem 2.6.3]{vershynin2018high}. Thus, by generic chaining \eqref{eq:generic-chaining},
\[  \sup_{\delta \in \sTheta_s} \abs{\nu' \varepsilon \delta} \le c \max_j \norm{\nu' \varepsilon_{\cdot j}}_{\psi_2}\set{\width(\sTheta_s) + u\rad(\sTheta_s)} \]
with conditional probability $1-2\exp(-u^2)$. Furthermore, as 
$\norm{\nu' \varepsilon_{\cdot j}}_{\psi_2} \le \norm{\Sigmacol^{1/2} \nu} 
						\norm{\Sigmacol^{-1/2}\varepsilon_{\cdot j}}_{\psi_2}$,
it follows that with probability $1-2\exp(-u^2)$,
\[ \sup_{\delta \in \sTheta_s} \abs{\nu' \varepsilon \delta} \le c K \norm{\Sigmacol^{1/2} \nu}\set{\width(\sTheta_s) + u\rad(\sTheta_s)}. \]
To bound $\norm{\Sigmacol^{1/2} \nu}$, we use generic chaining again. 
Letting $B:=\set{y \in \R^n : \norm{y} \le 1}$, 
\[ \norm{\Sigmacol^{1/2} \nu} = \sup_{y \in B} y'\Sigmacol^{1/2}\Sigmanu^{1/2} \Sigmanu^{-1/2}\nu
\le c\norm{\Sigmanu^{-1/2} \nu}_{\psi_2} \set{\width(\Sigmanu^{1/2}\Sigmacol^{1/2} B) + u \rad(\Sigmanu^{1/2}\Sigmacol^{1/2}  B)} \]
with probability $1-2\exp(-u^2)$. 
Furthermore, as $K \ge \norm{\Sigmanu^{-1/2} \nu}_{\psi_2}$,
$\width(B) \le cn^{1/2}$, and $\rad(B)=1$, this implies the simplified bound 
$\norm{\Sigmacol^{1/2}\nu} \le cK\norm{\Sigmacol}^{1/2}\norm{\Sigmanu}^{1/2}(n^{1/2} + u)$.
Via the union bound, it follows that with probability $1-4\exp(-u^2)$, 
\begin{align*}
\sup_{\delta \in \sTheta_s} \abs{\nu' \varepsilon \delta} 
&\le c K^2 \norm{\Sigmacol}^{1/2}\norm{\Sigmanu}^{1/2}(n^{1/2} + u)\set{\width(\sTheta_s) + u\rad(\sTheta_s)}.
\end{align*}
Combining this with our bound on the first term, it follows that with probability $1-6\exp(-u^2)$ via union bound,
\begin{align*} 
&\sup_{\delta \in \sTheta_s} \abs*{(\varepsilon\stheta - \nu)' \varepsilon\delta - \E (\varepsilon\stheta - \nu)' \varepsilon\delta} \\ 
&\le cK^2\norm{\Sigmacol}^{1/2}(\norm{\Sigmacol}^{1/2}\norm{\stheta} + \norm{\Sigmanu}^{1/2})(n^{1/2} + u)(\width(\Theta_s) + u\rad(\sTheta_s)) \\ 
&\le cK^2(\norm{\Sigmacol}^{1/2}\norm{\stheta} + \norm{\Sigmanu}^{1/2})(n^{1/2} + u)\set{\width_{\Sigma}(\Theta_s) + u \norm{\Sigmacol}^{1/2}\rad(\sTheta_s)}. 
\end{align*}

Taking $u= \min\set{v\width_{\Sigma}(\sTheta_s)/\norm{\Sigmacol}^{1/2}\rad(\sTheta_s), n^{1/2}}$, it follows that 
\begin{align*}
&\sup_{\delta \in \sTheta_s} \abs*{(\varepsilon\stheta - \nu)' \varepsilon\delta - \E (\varepsilon\stheta - \nu)' \varepsilon\delta} \le
 cv K^2 (n/p_{eff,\Sigma})^{1/2} \width_{\Sigma}(\sTheta_s)
\end{align*}
with probability $1-6\exp[\min\set{v^2 d_{\Sigma}(\sTheta_s), n}]$. 

\subsection{The case with intercept.}
\label{sec:with-intercept}
We consider a variant of this regression problem for $\bar X = \bar A + \bar \varepsilon$, where $\bar A$ and $\bar \varepsilon$ 
are augmented with columns of ones and zeros respectively. As our result follows from the bounds \eqref{eq:quadratic-signal-term}-\eqref{eq:quadratic-noise-term},
it is sufficient to show that the same bounds hold in the augmented problem. We will
write $[v; w] \in \R^{m+n}$ denoting the concatenation of vectors $v \in \R^m$ and $w \in \R^n$
and $\bar \delta$ referring to $[\theta_0 - \tilde \theta_0;\ \theta - \tilde\theta] \in \R^{1+p}$.

As $\bar \varepsilon [x_0; x] = \varepsilon x$ for any $x_0 \in \R$ and $x \in \R^p$, 
the bounds \eqref{eq:oracle-error-term} and \eqref{eq:quadratic-noise-term} are implied by the
corresponding bounds in the case without intercept if we substitute $\bar A [\tilde \theta_0; \tilde\theta] - b$
for $A\tilde\theta - b$ in \eqref{eq:oracle-error-term}.

We can show that \eqref{eq:quadratic-signal-term} need not be changed
for the augmented problem by working with the decomposition used in its proof:
\[ \norm{\bar X \bar\delta}^2 = \norm{\bar A \bar \delta}^2 + \norm{\bar\varepsilon\bar\delta}^2 + 2 (\bar A \bar \delta)' \bar \varepsilon \bar\delta. \]
The second term is equal to $\norm{\varepsilon\delta}^2$, so it satisfies the same lower bound as in the case without intercept.
And we can bound the second via Chevet's inequality as in the case without intercept; in the resulting bound  
$\width(A \sTheta_{s,r})$ is replaced by an analog $\width(\bar A \bsTheta_{s,r})$ 
with $\bsTheta_{s,r} = \set{[\delta_0; \delta] \in \R \times \sTheta : \norm{\Sigmarow^{1/2} \delta} \le s, \norm{\bar A [\delta_0; \delta]} \le r}$. 
We will show below that $\width(\bar  A \bsTheta_{s,r})$ satisfies the same bound that $\width(A \sTheta_{s,r})$ does in the case without intercept,
so we need not change \eqref{eq:quadratic-signal-term}.
 
Finally, the argument used to derive the bound \eqref{eq:low-rank-term} yields an analog in which again $\width(A \sTheta_{s,r})$ is replaced by 
 $\width(\bar A \bsTheta_{s,r})$, so again we not change \eqref{eq:low-rank-term} for the case with intercept.

We conclude by showing that $\width(\bar A \bsTheta_{s,r}) \le c\sqrt{R}r + \sigma_{R+1}(A) \width(\sTheta_{s})$,
matching our bound on $\width(A\sTheta_{s,r})$ from the case without intercept. 
For every $(\delta_0, \delta) \in \bsTheta_{s,r}$, 
$\delta_0 + A \delta = \delta_0' + A\delta  + \delta_0 - \delta_0'$ for $\delta_0' = -n^{-1}1_n' A\delta$.
Here $\norm{\delta_0' + A\delta} \le \norm{\delta_0 + A\delta} \le r$ because $\delta_0'=\argmin_{\delta_0''}\norm{\delta_0'' + A\delta}$;
furthermore, $\abs{\delta_0 - \delta_0'} \le 2n^{-1/2}r$, as $(\delta_0 - \delta_0')1_n = (\delta_0 + A \bar\delta) - (\delta_0' + A\delta)$
and therefore by the triangle inequality  $\abs{\delta_0 - \delta_0'}\norm{1_n} \le \norm{\delta_0 + A \bar\delta} + \norm{\delta_0' + A\delta} \le 2r$.
As we can write $\delta_0' + A\delta$ equivalently as $\tilde A\delta$ for $\tilde A = (I-n^{-1}1_n 1_n')A$, it follows 
that $\bar  A \bsTheta_{s,r} \subseteq n^{-1/2}1_n [-2r, 2r] + \tilde A \tilde\Theta^{\star}_{s,r}$  
where $\tilde \Theta^{\star}_{s,r} = \set{ \delta \in \sTheta : \norm{\tilde A \delta} \le r, \norm{\Sigmarow \delta} \le s}$
and therefore that $\width(\bar  A \bsTheta_{s,r}) \le \width( n^{-1/2}1_n [-2r, 2r] ) + \width(\tilde A \tilde\Theta^{\star}_{s,r})$.
The second term here has the same form as the quantity $\width(A \sTheta_{s,r})$ 
we considered in the case without intercept, with $\tilde A$ replacing $A$,
and via our arguments from Section~\ref{sec:third-term},
it is bounded by  $c \sqrt{R}r + \sigma_{R+1}(\tilde A) \width(\sTheta_{s})$ for any nonnegative integer $R$. 
And the first term here, $\width( n^{-1/2}1_n [-2r, 2r] )$, is bounded by $cr$ and therefore dominated by
our bound on the second term: because the distribution of a standard gaussian vector is rotation invariant 
and the unit vector $n^{-1/2}1_n$ can be rotated into the first standard basis vector $e_1$, 
$\width( n^{-1/2}1_n [-2r, 2r] )=\width(e_1 [-2r,2r]) = 2r \E \abs{g_1}$ for a standard normal random variable $g_1$. 
Observing that $\sigma_{R+1}(\tilde A) \le \sigma_{R+1}(A)$ because $\tilde A$ is a projection matrix times $A$, 
we get our claimed bound.

\section{Lemmas}

In this section, we collect lemmas used in the proof above.
Throughout, we will take $\xi$ to be a random matrix with independent rows $\xi_1 \ldots \xi_m$.
The first, a version of Chevet's inequality, bounds inner products through a random matrix.
\begin{lemma}
\label{lemm:chevet}
Let $\xi$ be a real $m \times n$ matrix with independent mean zero rows satisfying $\norm{\xi_i}_{\psi_2} \le K$,
and $V$ and $W$ be subsets of $\R^m$ and $\R^n$ containing zero. With probability $1-c\exp(-u^2)$,
\[ \sup_{v \in V, w \in W} v' \xi w \le cK[\width(V)\rad(W) + \rad(V)\width(W) + u \rad(V)\rad(W)]. \]
\end{lemma}
\noindent Because $w' \xi' v = v' \xi w$ and this bound is symmetric in $V$ and $W$, 
it also applies in the case that $\xi$ has independent columns.

The second, which we will use to bound quadratic forms $v'\xi \xi'v$,
is a variant of \citet[Theorem 6.5]{dirksen2015tail}, which is itself a refinement of a bound of 
\citet*{krahmer2014suprema}. For this bound and a subsequent one, we require that the squared norm of each row of $\xi$ 
concentrates around its mean essentially like (in the case $K=1$) that of a gaussian vector with the same covariance would, satisfying a Bernstein-type inequality
\begin{equation}
\label{eq:norm-bernstein}
P\p{\abs{\norm{\xi_i}^2 - \E \norm{\xi_i}^2 } \ge u} \le c\exp\p{-c \min\p{\frac{u^2}{K^4 \E\norm{\xi_i}^2},\  \frac{u}{K^2 \norm{\Sigma_i}}}}
 \ \text{ where } \ \Sigma_i = \E \xi_i \xi_i'.
\end{equation}
This condition is not implied by subgaussianty for any reasonable constant $K$. 
For example, the squared norm of the subgaussian vector $xg$, where $g$ is a gaussian vector and $x$ is Bernoulli($1/2$), does not concentrate around its mean at all. On the other hand, the Hanson-Wright inequality implies that it will be satisfied if the components of $\xi_i$ are mixtures of independent subgaussian random variables \citep[e.g.,][Theorem 6.2.1]{vershynin2018high}.

\begin{lemma}
\label{lemm:norm-of-sum-bound}
Let $\calA$ be a subset of $\R^m$ and $\xi \in R^{m \times n}$ be a matrix with independent mean-zero rows 
with common covariance matrix $\Sigma$, subgaussian norm $\norm{\xi_i\Sigma^{-1/2}}_{\psi_2} \le K$ for $K \ge 1$, and norms satisfying \eqref{eq:norm-bernstein}.
With probability $1-2\exp(-u^2)$ for $u \ge 1$,
\begin{align*}
&\sup_{a \in \calA} \abs*{\norm*{a'\xi}^2 - \E \norm*{a' \xi}^2 } \\ 
&\le cK^2 \sqb{ \norm{\Sigma}^{1/2} \trace^{1/2}(\Sigma) \width(\calA) \rad(\calA) + \norm{\Sigma} \width^2(\calA) + \norm{\Sigma}\rad^2(\calA)(n^{1/2}u + u^2)}. 
\end{align*}
\end{lemma}
\noindent Typically the dominant term will be the first, a multiple of $\rad(\calA)$.
Using a simple peeling argument, we refine the bound so the first term scales with $\norm{a}$ itself.
\begin{corollary}
\label{cor:norm-of-sum-bound}
In the setting of Lemma~\ref{lemm:norm-of-sum-bound}, with probability $1-2\exp(-u^2)$ for $u \ge 1$,
all $a \in \calA$ satisfy
\begin{align*} 
&\abs*{\norm*{a'\xi}^2 - \E \norm*{a' \xi}^2 } \\
&\le cK^2\sqb{ \norm{\Sigma}^{1/2} \trace^{1/2}(\Sigma) \width(\calA) \norm{a} + \norm{\Sigma} \width^2(\calA) + \norm{\Sigma}\norm{a}\rad(\calA)(n^{1/2}u + u^2) }. 
\end{align*}
\end{corollary}

The third is a variant for forms like $v'\xi \xi'w$.
\begin{lemma}
\label{lemm:weighted-inner-product-bound}
Let $\xi \in \R^{m \times n}$ be a matrix with independent mean-zero rows with common covariance matrix $\Sigma$, $\norm{\xi_i \Sigma^{-1/2}}_{\psi_2} \le K$ for $K \ge 1$,
and norms satisfying \eqref{eq:norm-bernstein}.
Let $v \in \R^m$ be a fixed vector and $W$ be a subset of $\R^m$. With probability $1-2\exp(-u^2)$ for $u \ge 1$,
\[ \sup_{w \in W} \abs{v'(\xi\xi' - \E \xi \xi')w} \le cK^2\norm{\Sigma} \norm{v} (\sqrt{n}+u)(\width(W) + u\rad(W)). \] 
\end{lemma}

\subsection{Proof of Lemma~\ref{lemm:chevet}}
Our claim is a slight variation on \citet[Proposition 2.6.1]{vershynin2018high}. Its proof needs only minor modifications.
The proof is based on a generic chaining bound \citep[Theorem 8.5.5]{vershynin2018high}. 
If $Z_t$ is a stochastic process indexed by elements $t$ of a metric space $(T,d)$, and its increments satisfy
$\norm{Z_t - Z_s}_{\psi_2} \le K d(t,s)$ for all $t,s \in T$, then
\begin{equation}
\label{eq:generic-chaining} 
\sup_{t,s \in T} \smallabs{Z_t - Z_s} \le cK (\gamma_2(T,d) + u \diam(T,d))\ \text{ with probability }\ 1-2\exp(-u^2) 
\end{equation}
In our setting, we take $t=(v,w) \in V \times W$ and define $Z_{vw} = v' \xi w$. We establish our increment bound via the triangle inequality and Hoeffding's.
Because $Z_{vw} - Z_{xy} = (v-x)' \xi w + x' \xi (w-y)$,
\begin{align*} 
\norm{Z_{vw} - Z_{xy}}_{\psi_2}
&\le \norm{(v-x)' \xi w}_{\psi_2} + \norm{x'\xi(w-y)}_{\psi_2} \\
&\le c\norm{v-x}\max_i \norm{\xi_i w}_{\psi_2} + c\norm{x}\max_i\norm{\xi_i (w-y)}_{\psi_2} \\
&\le c\max_{i}\norm{\xi_i}_{\psi_2}[\norm{v-x}\norm{w} + \norm{x}\norm{w-y}] \\
&\le cK[\norm{v-x}\rad(W) + \norm{w-y}\rad(V)] \\
&\le cK \sqrt{ \norm{v-x}^2\rad(W)^2 + \norm{w-y}^2\rad(V)^2  }.
\end{align*}
Letting $d^2((v,w), (x,y))$ be the expression under the square root, generic chaining gives 
\[ \sup_{t,s \in T} \norm{Z_{vw} - Z_{xy}}_{\psi_2} \le cK (\gamma_2(V \times W, d) + u \rad(V)\rad(W)). \]
Here we've used the property that $\diam(V \times W, d)^2 = \diam(V)^2\rad(W)^2 + \diam(V)^2\rad(W)^2 \le 8\rad(V)^2\rad(W)^2$.
To bound $\gamma_2(V \times W, d)$, as in the proof of \citet[Proposition 2.6.1]{vershynin2018high},
we use Talagrand's majorizing measures theorem \citep[Theorem 2.4.1]{talagrand2014upper}:
$\gamma_2(T, d) \le c \E \sup_{t \in T} Y_{t}$ for a gaussian process $Y_t$ with
increments satisfying $\norm{Y_t - Y_s}_{\psi_2}=d(t,s)$. The claimed bound follows 
by considering the gaussian process $Y_{vw} = (g'v)\rad(W) + (h'w)\rad(V)$ for independent gaussian vectors $g,h$.

\subsection{A Decoupling Inequality}
The next several lemmas will be based on a decoupling inequality for quadratic forms $\sum_{ij} B_{ij} \inner{ \xi_i, \xi_j}$
in independent subgaussian vectors with norms $\norm{\xi_i}$ that concentrate around their means essentially like those norm of a gaussian vector would. 
Its proof is based largely on a reduction to a decoupling inequality for quadratic forms in gaussian random variables.

\begin{lemma}
\label{lemm:decoupling-quadratic}
Let $\xi_1 \ldots \xi_n$ be a sequence of independent mean-zero random vectors with 
covariance matrices $\Sigma_1 \ldots \Sigma_n$ and norms satisfying a Bernstein-type concentration inequality 
\[ P\p{\abs{\norm{\xi_i}^2 - \E \norm{\xi_i}^2 } \ge u} \le c\exp\p{-c \min\p{\frac{u^2}{K^4 \E\norm{\xi_i}^2},\  \frac{u}{K^2\norm{\Sigma_i}}}}. \]
Let $g_1 \ldots g_n$ be a sequence of independent mean-zero gaussian vectors with the covariance matrices $\Sigma_1 \ldots \Sigma_n$,
and let $\tilde \xi_1 \ldots \tilde \xi_n$ and $\tilde g_1 \ldots \tilde g_n$ be independent copies.
Then for any set $\mathcal{B}$ of $n \times n$ matrices
and any Orlicz norm $\norm{Z}_{F} := \inf\set{ s > 0 : \E[F(\abs{Z}/s)] \le 1}$,
\[ \norm*{\sup_{B \in \mathcal{B}}\abs{ \sum_{i,j} B_{ij} \inner{\xi_i, \xi_j} - \E \sum_{i,j} B_{ij} \inner{\xi_i, \xi_j}}}_{F} \le 
    c \max\p{\norm*{\sup_{B \in \mathcal{B}} \abs{\sum_{i,j} B_{ij} \inner{\xi_i, \tilde \xi_j}}}_{F},\   
             K^2 \norm*{\sup_{B \in \mathcal{B}} \abs{\sum_{i,j} B_{ij} \inner{g_i, \tilde g_i}}}_{F}}. \]
\end{lemma}

\noindent To prove it, we bound the off-diagonal and diagonal terms separately, as via triangle inequality,
\[ \norm*{\sup_{B \in \mathcal{B}} \abs{\sum_{i,j} B_{ij} \inner{\xi_i, \xi_j} - \E \sum_{i,j} B_{ij} \inner{\xi_i, \xi_j}}}_{F} \le 
    \norm*{\sup_{B \in \mathcal{B}} \abs{\sum_{i,j : i \neq j} B_{ij} \inner{\xi_i, \xi_j}}}_{F} +
    \norm*{\sup_{B \in \mathcal{B}} \abs{\sum_{i} B_{ii} [\norm{\xi_i}^2 - \E\norm{\xi_i}^2]}}_{F}. \]

A straightforward generalization of a standard decoupling inequality \citep[e.g.,][Theorem 6.1.1]{vershynin2018high}
suffices to bound the sum of off-diagonal terms. For any convex function $F$,
\begin{equation}
\label{eq:standard-decoupling}
 \E F( \sup_{B \in \mathcal{B}} \abs{\sum_{ij : i \neq j} B_{ij} \inner{\xi_i, \xi_j}}) \le 
   \E F(4 \sup_{B \in \mathcal{B}} \abs{\sum_{ij} B_{ij} \inner{\xi_i, \tilde \xi_j}}), 
\end{equation}
so the associated Orlicz norm of the sum of off-diagonal terms is no more than four times that of the decoupled variant.
We will not prove it here, as the proof in \citet{vershynin2018high} generalizes trivially.

To bound the sum of diagonal terms, we compare it to an analog in which $\xi_i$ is replaced by a mean-zero gaussian vector $g_i$ with the same covariance.
Our argument, which largely follows one of \citet[proof of Theorem 6.5]{dirksen2015tail}, is based on symmetrization and contraction \citep[Lemmas 6.3 and 4.6]{ledoux2013probability}. 
Let $Z_i = \norm{\xi_i}^2 - \E \norm{\xi_i}^2$.
For an iid rademacher sequence $\varepsilon_1 \ldots \varepsilon_n$ and vectors $b$ ranging over any set,
\begin{equation}
\label{eq:contracted}
 \E F(\sup_{b} \abs{\sum_i b_i Z_i }) \le \E F(2 \sup_{b} \abs{\sum_i b_i \varepsilon_i Z_i }) \le 
    \E F(2M \sup_{b} \abs{\sum_i b_i \varepsilon_i Z_i^g) })
\end{equation}
for any $Z_i^g$ satisfying $P(\abs{Z_i} \ge u) \le M P(\abs{Z_i^g} \ge u)$. 
In particular, this holds for $Z_i^g = cK^2(\norm{g_i}^2 - \E \norm{g_i}^2)$,
as for a large-enough choice of constant $c$, our assumed upper bound on the tail of $\abs{Z_i}$ is no larger than 
a constant multple of our lower bound
\[ P(\abs{Z_i^{g}} \ge u/cK^2) \ge cP( \inner{g_i, \tilde g_i} \ge u/cK^2) \ge c\exp\p{-c \min\p{\frac{u^2}{K^4 \E\norm{\xi_i}^2},\ \frac{u}{K^2\norm{\Sigma_i}}}}. \]
The first lower bound here is via a decoupling inequality for gaussian chaos \citep[Theorem 4.2.7]{de2012decoupling}
and the second via a lower bound on the tail of decoupled gaussian chaos \citep[Corollary 1]{latala2006estimates}. 
Thus, via \eqref{eq:contracted} and then desymmetrization \citep[Lemma 6.3]{ledoux2013probability},
\[  \E F(\sup_{b} \abs{\sum_i b_i Z_i} ) \le \E F( c\sup_{b} \abs{\sum_i b_i \varepsilon_i Z_i^g })
                                         \le \E F( c\sup_{b} \abs{\sum_i b_i Z_i^g}). \]
We conclude via triangle inequality. Letting $B$ range over $\mathcal{B}$, 
\begin{align*}
\norm{\sup_{B} \abs{\sum_i B_{ii} Z_i}}_{F} &\le cK^2\norm{\sup_{B} \abs{\sum_i B_{ii} Z_i^g}}_{F} \\
&= cK^2\norm{\sup_{B} \abs{\sum_{ij} B_{ij} [\inner{g_i, g_j} - \E \inner{g_i,g_j}]  - \sum_{ij : i\neq j} B_{ij} \inner{g_i,g_j}}}_F \\
&\le cK^2\norm{\sup_{B} \abs{\sum_{ij} B_{ij} \inner{g_i, g_j} - \E \inner{g_i,g_j}}}_F +
     cK^2\norm{\sup_{B} \abs{\sum_{ij : i\neq j} B_{ij} \inner{g_i,g_j}}}_F \\
&\le cK^2\norm{\sup_{B} \abs{\sum_{ij} B_{ij} \inner{g_i, \tilde g_j}}}_{F}.
\end{align*}
In the last step, the term involving diagonal elements is bounded by a multiple of the decoupled analog
via a decoupling inequality for gaussian chaos \citep[Theorem 4.2.7]{de2012decoupling}
and the sum without diagonal terms via our more standard decoupling inequality \eqref{eq:standard-decoupling}.

\subsection{Proof of Lemma~\ref{lemm:norm-of-sum-bound}}

We base our proof on that of \citet[Theorem 6.5]{dirksen2015tail}, which uses a generic chaining argument.
Define $\gamma_{2,p}(\calA, d) := \inf_{\calA_n} \sup_{a \in \calA}\sum_{n > \ell} 2^{n/2}d(a,\calA_n)$ where 
$d(a,a'):=\norm{a-a'}$ and the infimum is taken over admissible sequences. We will prove the moment bound 
\begin{equation}
\label{eq:norm-of-sum-moment}
\begin{aligned} 
&\norm*{\sup_{a \in \calA} \abs*{\norm*{\sum_i a_i \xi_i}^2 - \E \norm*{\sum_i a_i \xi_i}^2 }}_{L_p} \\
&\le c K^2 \norm{\Sigma} \gamma_{2,p}(\calA)^2 + cK^2 \gamma_{2,p}(\calA)\norm{\Sigma}^{1/2} \trace^{1/2}(\Sigma) \rad(\calA) + cK^2\norm{\Sigma}(p + (pn)^{1/2})\rad^2(\calA).
\end{aligned}
\end{equation}
The claimed tail bound follows \citep[Lemma A.1]{dirksen2015tail}.

Let $(\calA_n)_{n \ge 0}$ be an admissible sequence that comes within a constant multiplicative factor of this infimum, 
let $a^{n} := \argmin_{a \in \calA_n}d(a,a')$, and let $\ell := \floor{\log(p)}$.
Letting suprema over $a$ be implicitly be taken over $a \in \calA$, we write
\begin{equation}
\begin{aligned} 
& \sup_{a} \abs*{\norm*{\sum_i a_i \xi_i}^2 - \E \norm*{\sum_i a_i \xi_i}^2 } && =: S \\
&= \sup_{a} \abs*{\sum_{ij}a_i a_j \p{\inner{\xi_i,\xi_j} - \E \sum_{ij} a_i a_j \inner{\xi_i, \xi_j}}} && \\
&\le \sup_{a} \abs*{\sum_{ij}a_i^{\ell} a_j^{\ell} \p{\inner{\xi_i,\xi_j} - \E \inner{\xi_i, \xi_j}}} && =: B \\
 &+ \sup_{a} \abs*{\sum_{ij} \p{a_i a_j - a_i^{\ell} a_j^{\ell}} \p{\inner{\xi_i, \xi_j} - \E \inner{\xi_i, \xi_j}}} && =: C.
\end{aligned}
\end{equation}
Our overall strategy will be to show that
$\norm{B}_{L_p} + \norm{C}_{L_p} \le 2\alpha \norm{S}_{L_p}^{1/2} + \beta$ 
for some constants $\alpha$ and $\beta$, implying that $\norm{S}_{L_p} \le 2\alpha \norm{S}_{L_p}^{1/2} + \beta$.
Because this quadratic inequality is satisfied only if 
$\norm{S}_{L_p}^{1/2} \le (\alpha + \sqrt{\alpha^2 + \beta})$, 
this implies the bound 
\[ \norm{S}_{L_p} \le \p{\alpha + \sqrt{\alpha^2 + \beta}}^2  \le 4(\alpha^2 + \beta). \]

There is one wrinkle. We will show a bound of the form $\norm{C}_{L_p} \le \alpha \max(\norm{S}_{L_p}^{1/2}, K^2 \norm{S_g}_{L_p}^{1/2}) + \beta$ 
where $S_g$ is an analog of $S$ in which the vectors $\xi_i$ are replaced with a independent gaussian vectors with the same mean and covariance.
Thus, instead of the quadratic inequality above, we get $\norm{S}_{L_p} \le 2\alpha \max(\norm{S}_{L_p}, K^2 \norm{S_g})^{1/2} + \beta$. 
Considering the case that $\xi_i$ \emph{are} gaussian, our argument implies that  
$\norm{S_g}_{L_p} \le 2\alpha \norm{S_g}_{L_p}^{1/2} + \beta$, and it follows that in general,
$\max(\norm{S}_{L_p}, K^2 \norm{S_g}_{L_p}) \le 2\alpha \max(\norm{S}_{L_p}, K^2 \norm{S_g}_{L_p})^{1/2} + \beta$.
Thus, the bound derived above holds for $\max(\norm{S}_{L_p}, K^2 \norm{S_g}_{L_p})$ and therefore for $\norm{S}_{L_p}$.

To complete our argument, we will prove bounds on $\norm{B}_{L_p}$ and $\norm{C}_{L_p}$ that
imply \eqref{eq:norm-of-sum-moment}. Specifically, we can take $\alpha = cK \norm{\Sigma}^{1/2} \gamma_{2,p}(\calA)$
and $\beta = \alpha K \trace^{1/2}(\Sigma) \rad(\calA) + cK^2\norm{\Sigma}(p + (pn)^{1/2})\rad^2(\calA)$,
with the first term in $\beta$ coming from $C$ and the second from $B$.

\subsubsection{Bounding $\norm{B}_{L_p}$}
\label{sec:bounding-B}
Via Lemma~\ref{lemm:decoupling-quadratic}, it suffices to bound $\norm{\tilde B}$ for a decoupled variant of $B$,
\[ \tilde B := \sup_{a} \abs*{\sum_{ij}a_i^{\ell} a_j^{\ell} \p{\inner{\xi_i, \tilde \xi_j} - \E \inner{\xi_i, \tilde \xi_j}}}, \]
and for a variant $K^2 \tilde B_g$ in which $\xi_i,\tilde\xi_i$ are replaced by gaussian analogs $K g_i$ and $K\tilde g_i$.
We will consider $K^2 \tilde B_g$ after treating $\tilde B$.

We bound $\tilde B$ via a variant of the Hanson-Wright inequality \citep[Exercise 6.2.7]{vershynin2018high}.
\[ \abs*{\sum_{ij} a_i a_j \inner{\xi_i, \tilde \xi_j}} \le u
\ \text{ with probability }\ 1-2\exp\p{-c\min\p{\frac{u^2}{(K_{\Sigma}^4 n^{1/2} \norm{a a'}_{F})^2}, \frac{u}{K_{\Sigma}^2 \norm{a a'}}}}, \]
when $\norm{\xi_i}_{\psi_2} \le cK_{\Sigma}$. As $\norm{\xi_i}_{\psi_2}\le \norm{\xi_i \Sigma^{-1/2}}_{\psi_2}\norm{\Sigma^{1/2}} \le K\norm{\Sigma}^{1/2}$,
this holds for $K_{\Sigma}=K\norm{\Sigma}^{1/2}$.

This implies the moment bound \citep[Lemma A.1]{dirksen2015tail} 
\[ \norm*{\sum_{ij} a_i a_j \inner{\xi_i, \tilde \xi_j}}_{L_p} \le  cK^2 \norm{\Sigma} \p{p\norm{a a'} + p^{1/2}n^{1/2}\norm{a a'}_{F}} = cK^2 \norm{\Sigma} (p + p^{1/2}n^{1/2})\norm{a}^2 \] 
using the identity $\norm{a a'}_{F} = \norm{a a'} = \norm{a}^2$ to simplify. Call this bound $\beta_B(a)$.
Because there are only $2^{2^{\ell}} \le 2^p$ elements $a^{\ell} \in \calA_{\ell}$, at the cost of an additional factor of $2$, we get a uniform bound: 
\[
\norm{\tilde B}_{L_p}
\le \sqb{\sum_{a^{\ell} \in \calA_{\ell}}\E \abs*{ \sum_{ij} a_i^{\ell} a_j^{\ell} \p{\inner{\xi_i, \xi_j} - \E\inner{\xi_i, \xi_j}}}^p}^{1/p} 
\le [2^p \sup_{a}\beta_B^p(a)]^{1/p}.
\]
This last quantity $2\sup_{a}\beta_B(a)$ is the second term in $\beta$.

The only property of $\xi_i$ that we use here is that its subgaussian norm is bounded by $cK\norm{\Sigma}^{1/2}$.
Because $Kg_i$ also has this property, the bound above also applies to $K^2 \tilde B_g$.

\subsubsection{Bounding $\norm{C}_{L_p}$}
\label{sec:bounding-C}
In this section, we will prove the bound $\norm{C}_{L_p} \le \alpha (\norm{S}_{L_p}^{1/2} + \norm{\Sigma}_F \rad(\calA))$.
Via Lemma~\ref{lemm:decoupling-quadratic}, it suffices to bound a decoupled variant,
\[ \tilde{C} := \sup_{a}\tilde{C}(a) \ \text{ where }\ \tilde{C}(a) := \abs*{\sum_{ij} (a_i a_j - a_i^{\ell}a_j^{\ell}) \inner{\xi_i, \tilde\xi_j}}, \]
as well as variant $K^2 \tilde C_g$ in which $\xi_i,\tilde\xi_i$ are replaced by gaussian analogs $Kg_i,K\tilde g_i$. 

To bound $\tilde C$, we decompose $\tilde C(a)$ into two terms, 
and then decompose each into a sum over `links' $a^{n}-a^{n-1}$.
\begin{align*}
&\sum_{ij} (a_i a_j - a_i^{\ell}a_j^{\ell}) \inner{\xi_i, \tilde\xi_j} \\
&= \sum_{ij} (a_i - a_i^{\ell})a_j \inner{\xi_i, \tilde \xi_j} + \sum_{ij} a_i^{\ell}(a_j - a_j^{\ell})\inner{\xi_i, \tilde \xi_i} \\
&= \sum_{n > \ell}\sum_{i} \inner{\xi_i, (a_i^{n}-a_i^{n-1}) \sum_j a_j \tilde \xi_j} 
+  \sum_{n > \ell}\sum_{j} \inner{ (a_j^n - a_j^{n-1}) \sum_i a_{i}^{\ell} \xi_i , \tilde \xi_j} \\
&=: \tilde C_1(a) + \tilde C_2(a)
\end{align*}

We'll focus on $\tilde C_1$. Conditional on $\tilde \xi$, terms
$\inner{\xi_i, (a_i^{n}-a_i^{n-1})\sum_j a_j \tilde \xi_j}$ are independent with 
\[ \norm*{\inner{\xi_i, (a_i^{n}-a_i^{n-1}) \sum_j a_j \tilde \xi_j}}_{\psi_2} 
\le \norm*{\xi_i}_{\psi_2}\norm*{(a_i^{n}-a_i^{n-1}) \sum_j a_j \tilde \xi_j} \le K\norm{\Sigma}^{1/2} \abs{a_i^{n}-a_i^{n-1}} \norm*{\sum_j a_j \tilde \xi_j}. \]
Via Hoeffding's inequality, it follows that with probability $1-2\exp(-u^2 2^n)$, 
\[ \abs*{ \sum_{i} \inner{\xi_i, (a_i^{n}-a_i^{n-1}) \sum_j a_j \tilde \xi_j}} < c u 2^{n/2} K\norm{\Sigma}^{1/2} \norm{a^n - a^{n-1}} \norm*{\sum_j a_j \tilde \xi_j}, \] 
And as there are only $2^{2^n}$ possible values of $a_i^{n}-a_i^{n-1}$ for $a \in \calA$, by the union bound this will hold uniformly with reasonably high probability.
In particular, it holds simultaneously for all $a \in \calA$ and all $n > \ell$ on an event of probability $1-c\exp(-pu^2/4)$ for $u \ge \sqrt{2}$ \citep[Lemma A.4]{dirksen2015tail}. On this event, 
\[ \sup_{a} \abs*{\tilde C_1(a)} \le cK\norm{\Sigma}^{1/2} u  \sup_{a} \norm*{\sum_j a_j \tilde \xi_j} \sup_{a} \sum_{n > \ell} 2^{n/2} \norm{a^n - a^{n-1}}
                                 \le cK\norm{\Sigma}^{1/2} u  \sup_{a} \norm*{\sum_j a_j \tilde \xi_j} \gamma_{2,p}(\calA). \]

This conditional tail bound implies a conditional moment bound \citep[Lemma A.5]{dirksen2015tail} and consequently, averaging over $\tilde\xi$, an unconditional one:
\[
\norm*{\sup_{a}\abs*{\tilde C_1(a)}}_{L_p} \le cK\norm{\Sigma}^{1/2} \gamma_{2,p}(\calA)  \norm*{\sup_a \norm*{\sum_j a_j \tilde \xi_j}}_{L_p}.
\]
This holds with $\xi$ in place of $\tilde \xi$ because the two are identically distributed.
Analogously, $\norm{\sup_{a}\abs{\tilde C_2(a)}}_{L_p} \le cK \norm{\Sigma}^{1/2}\gamma_{2,p}(\calA) \norm{\sup_A \norm{\sum_j a_j \xi_j}}_{L_p}$, so by the triangle inequality,
$\norm{\tilde{C}}_{L_p} \le cK \norm{\Sigma}^{1/2} \gamma_{2,p}(\calA)$ $\norm{\sup_{a}\norm{\sum_j a_j \xi_j}}_{L_p}$.
Finally, via the triangle inequality and H\"older's,
\begin{align*} 
\sup_{a}\norm*{\sum_j a_j \xi_j}^2 
&\le \sup_{a} \abs*{\norm*{\sum_j a_j \xi_j}^2 - \E \norm*{\sum_j a_j \xi_j}^2} + \sup_{a} \E \norm*{\sum_j a_j \xi_j}^2 \\
&= S + \sup_{a} \sum_j a_j^2 \norm{\xi_j}_{L_2}^2 \le S + \trace(\Sigma) \rad^2(\calA). 
\end{align*}
It follows by the (quasi-)triangle inequality in $L_{p/2}$ that
\begin{align*}
\norm{\tilde{C}}_{L_p} 
&\le cK\norm{\Sigma}^{1/} \gamma_{2,p}(\calA) \p{ \E (S + \trace(\Sigma) \rad^2(\calA))^{p/2} }^{1/p}  \\
&\le cK\norm{\Sigma}^{1/2} \gamma_{2,p}(\calA) \p{ \norm{S}_{L_{p/2}} + \trace(\Sigma) \rad^2(\calA)}^{1/2} \\  
&\le cK\norm{\Sigma}^{1/2} \gamma_{2,p}(\calA) \p{ \norm{S}_{L_{p/2}}^{1/2} + \trace^{1/2}(\Sigma) \rad(\calA)} \\
&\le cK\norm{\Sigma}^{1/2} \gamma_{2,p}(\calA) \p{ \norm{S}_{L_{p}}^{1/2} + \trace^{1/2}(\Sigma) \rad(\calA)}.
\end{align*}

We'll now adapt this argument to bound $K^2 \tilde C_g$. We use two properties of $\xi_i$ to derive this bound. 
The first is that its subgaussian norm is bounded by $cK\norm{\Sigma}^{1/2}$,
a property that $K g_i$ also has. The second is that $\E \norm{\xi}^2 = \trace(\Sigma)$. To adapt our result,
it suffices to replace $\trace(\Sigma)$ with $\E \norm{K g_i}^2 = K^2 \trace(\Sigma)$. Because we've assumed $K \ge 1$,
including this factor of $K$ gives a bound on both $\tilde C$ and $K^2 \tilde C_g$.
This, combined with our decoupling result, establishes our claimed bound on $\norm{C}_{L_p}$.

\subsubsection{Proof of Corollary~\ref{cor:norm-of-sum-bound}}
$\calA$ can be decomposed as the union of $\calA_k := \set{ a \in \calA : \norm{a} \in [2^{-k}\rad(\calA),  \le  2^{1-k}\rad(\calA)]}$ for integers $k \ge 1$.
And with probability $1-2\exp(-u_k^2)$ for $u_k := (u^2+k)^{1/2}$, all $a \in \calA_k$ satisfy 
\begin{align*} 
&\abs*{\norm*{a'\xi}^2 - \E \norm*{a' \xi}^2 } \\
&\le cK^2\sqb{ \norm{\Sigma}^{1/2} \trace^{1/2}(\Sigma) \width(\calA_k) \rad(\calA_k) + \norm{\Sigma} \width^2(\calA_k) + \norm{\Sigma}\rad^2(\calA_k)(n^{1/2}u_k + u_k^2)} \\
&\le cK^2\sqb{ \norm{\Sigma}^{1/2} \trace^{1/2}(\Sigma) \width(\calA) \norm{a} + \norm{\Sigma} \width^2(\calA) + \norm{\Sigma}\norm{a}\rad(\calA)2^{-k}(n^{1/2}(u + \sqrt{k}) + u^2 + k) }\\
&\le cK^2 \sqb{ \norm{\Sigma}^{1/2} \trace^{1/2}(\Sigma) \width(\calA) \norm{a} + \norm{\Sigma} \width^2(\calA) + \norm{\Sigma}\norm{a}\rad(\calA)(n^{1/2}u + u^2) }.
\end{align*}
The first bound here is a direct application of Lemma~\ref{lemm:norm-of-sum-bound}. The second follows from the bounds
$\rad(\calA_k) \le 2\norm{a}$ for all $a \in \calA_k$, $\rad(\calA_k) \le 2^{1-k}\rad(\calA)$, and $\width(\calA_k) \le \width(\calA)$.
And the third follows from the bounds $2^{-k}\sqrt{k} \le 1 \le u$ and $2^{-k}k \le 1 \le u^2$.
Our claimed bound has the same right side, and because every $a \in \calA$ belongs to some $\calA_k$, 
it follows that it holds if the bound above is satisfied for all $k \ge 1$.
By the union bound, this happens with probability no smaller than 
$1-2\sum_{k \ge 1}\exp(-u_k^2) = 1-2\exp(-u^2)\sum_{k \ge 1}\exp(-k) = 1-2\exp(-u^2) (1/(e-1)) \ge 1-2\exp(-u^2)$.

\subsection{Proof of Lemma~\ref{lemm:weighted-inner-product-bound}}
Our claimed tail bound will be implied by moment bounds, so via Lemma~\ref{lemm:decoupling-quadratic},
it suffices to bound a decoupled version of our process. In particular,
letting $Z := \sup_{w \in W} v' (\xi \xi' - \E \xi \xi') w$
and $\tilde Z := \sup_{w \in W} v'\tilde\xi \xi' w$ be a decoupled version in which $\tilde \xi$ is an independent copy of $\xi$, 
$\norm{Z}_{L_p} \le c \norm{\tilde Z}_{L_p}$. To be more precise, it satisfies the bound 
$\norm{Z}_{L_p} \le c \max(\norm{\tilde Z}_{L_p}, K^2 \norm{\tilde Z_g}_{L_p})$ for a gaussian analog $K^2 Z_g$ of $\tilde Z$,
but the argument we use to bound $\tilde Z$ implies the same bound on $K^2 \tilde Z_g$, so we will not mention this in what remains.

To bound $\tilde Z$, we condition on $\tilde \xi$, fixing the realization of $y=v'\tilde \xi$,
 and bound this conditionally subgaussian process $y' \xi w$ by generic chaining \eqref{eq:generic-chaining}.
Because increments satisfy $\norm{y' \xi' w - y' \xi' x}_{\psi_2} \le \norm{y' \xi'}_{\psi_2}\norm{w-x}$ conditional on $\tilde \xi$,
$\sup_{w \in W} \abs{y' \xi' w} \le c \norm{y' \xi}_{\psi_2}(\width(W) + u \rad(W))$ with probability $1-2\exp(-u^2)$.
Furthermore, $\norm{y'\xi'}_{\psi_2} = \norm{y'\Sigma^{1/2} (\xi \Sigma^{-1/2})'} \le K \norm{y'\Sigma^{1/2}}$ via \eqref{eq:subgaussian-norm-mv}.

Now let $\tilde A = \tilde \xi \Sigma^{-1/2}$. 
To bound $\norm{y \Sigma^{1/2}} = \norm{v' \tilde A \Sigma} = \sup_{\norm{z} \le 1}v' \tilde A \Sigma z$, we use generic chaining essentially as we've just done.
Letting $B:=\set{z \in \R^n : \norm{z} \le 1}$, $\norm{y} \le c\norm{v' \tilde A}_{\psi_2} (\width(\Sigma B) + u\rad(\Sigma B))$
with probability $1-2\exp(-u^2)$. It follows that $\norm{y\Sigma^{1/2}} \le c K\norm{v}\norm{\Sigma}(\sqrt{n} + u)$ with probability $1-2\exp(-u^2)$.
This is a consequence of the bounds $\norm{v'\tilde A}_{\psi_2} \le K\norm{v}$,
$\rad(\Sigma B) \le \norm{\Sigma}\rad(B) = \norm{\Sigma}$,
and $\width(\Sigma B) \le \norm{\Sigma}\width(B) \le c\norm{\Sigma}\sqrt{n}$. 
The last of these follows from straightforward calculation:
letting $g \in \R^m$ be a vector of independent standard gaussians, $\width(B) := \E \sup_{\norm{z} \le 1}g'z = \E \norm{g} \le (\E \norm{g}^2)^{1/2} = \sqrt{n}$. 
Thus, with probability $1-4\exp(-u^2)$ by the union bound, 
\begin{align*} 
\tilde Z &\le cK^2 \norm{v} \norm{\Sigma} (\sqrt{n} + u) (\width(W) + u\rad(W)) \\
&= cK^2\norm{v} \norm{\Sigma} [\sqrt{n}\width(W) + u (\sqrt{n}\rad(W) + \width(W)) + u^2\rad(W)]. 
\end{align*}
To pass this Bernstein-type tail bound though our decoupling inequality, we observe that it implies moment bounds,
and that those moment bounds imply an equivalent (up to constant factor) tail bound \citep[Lemmas A.1 and A.2]{dirksen2015tail}. 
Furthermore, increasing the leading constant, a bound of this form holds with probability $1-2\exp(-u)$:
it holds with probability for $1-2\exp(-u')$ for $u' = u-\log(2) \ge cu$ when $u \ge 1$.

\section{Proofs for the Synthetic Control Estimator}
In this section, we prove Proposition~\ref{proposition:sc-bias-bound} and Corollary~\ref{corollary:sc-normality}.
We use the following lemma.
\begin{lemma}
\label{lemma:ridge-approx}
For any real matrix $A$ and vector $b$,
$\min_x \alpha^2 \norm{Ax - b}^2 + \beta^2 \norm{x}^2 = \alpha^2 \norm{S^{1/2} b}^2$
for $S=I-A(A'A + (\beta/\alpha)^2 I)^{-1}A'$. In terms of the SVD $A=\sum_k \sigma_k u_k v_k'$,
this is $\alpha^2 \sum_k (u_k' b)^2  / (1+ \sigma_k^2 \alpha^2/\beta^2)$.
\end{lemma}
\begin{proof}
Defining $\kappa^2=\beta^2/\alpha^2$,
 this is $\alpha^2$ times $\min_x \norm{Ax-b}^2 + \kappa^2 \norm{x}^2 = \min_x \norm{b}^2 - 2x'A'b + x'(A'A + \kappa^2 I)x$. 
Setting the derivative of the expression to zero, we solve for the minimizer $x=(A'A + \kappa^2 I)^{-1}A'b$ 
and the minimum $b'[I - A(A'A + \kappa^2 I)^{-1}A']b$, then multiply by $\alpha^2$.
In terms of the SVD of $A$, the singular value decomposition of the bracketed matrix is 
$\sum_k u_k u_k' \set{1 - \sigma_k^2/(\sigma_k^2 + \kappa^2)}=\sum_k u_k u_k' \kappa^2/(\sigma_k^2 + \kappa^2) = \sum_k u_k u_k' / (\sigma_k^2/\kappa^2 + 1).$
Expanding $b=\sum_k (u_k' b) u_k$, our claim follows.
\end{proof}

\subsection{Proof of Proposition~\ref{proposition:sc-bias-bound}} 
\label{sec:proof-of-sc-bias-bound}

We work with the error decomposition below.
\begin{equation}
\label{eq:sc-oracle-deviation-decomp}
\begin{aligned}
 \tilde \tau - \hat \tau 
&= x_e' (\hat \theta - \tilde \theta) \\
&= \lambda' A(\hat \theta - \tilde \theta) + (a_e - \lambda' A)(\hat \theta - \tilde \theta) \\ 
&+    \psi_{col}'\varepsilon (\hat \theta - \tilde\theta) + (\varepsilon_e - \psi_{col}'\varepsilon) (\hat \theta - \tilde\theta).
\end{aligned}
\end{equation}
We separately bound
the minimum of the first two terms over $\lambda$, the third term, and the fourth term. 
We begin with the first of these. On an event on which the bounds $\norm{A(\hat \theta - \tilde\theta)} \le \eta n^{1/2}s$
and $\sigma \norm{\hat \theta - \stheta} \le s$ (because $\Sigmarow^{1/2}=\sigma I$) of Theorem~\ref{theorem:rate-simplified} hold,
\begin{equation}
\begin{aligned}
&\min_{\lambda} \abs{\lambda' A(\hat \theta - \tilde \theta) + (a_e - \lambda' A)(\hat \theta - \tilde \theta)} \\ 
&\le \min_{\lambda} \set{\norm{\lambda} \norm{A(\hat \theta - \tilde\theta)} + \norm{a_e - \lambda'A} \norm{\hat \theta - \tilde\theta}} \\
&\le \min_{\lambda}\set{ \norm{\lambda} \eta n^{1/2}s +  \norm{a_e - \lambda'A}(s/\sigma) } \\
&\le \sqrt{2}s\min_{\lambda}\p*{ \norm{\lambda}^2 \eta^2 n + \norm{A' \lambda - a_e'}^2/\sigma^2}^{1/2} \\
&= \sqrt{2}(s/\sigma)\set*{\sum_k \frac{ (v_k' a_e )^2}{1+\sigma_k^2/(\sigma^2\eta^2 n)}}^{1/2},\ A=\sum_k \sigma_k u_k v_k'. 
\end{aligned}
\end{equation}
Here the penultimate expression is derived via the elementary inequality $(x+y)^2 \le 2(x^2+y^2)$
and the ultimate via Lemma~\ref{lemma:ridge-approx}.

In the third term of \eqref{eq:sc-oracle-deviation-decomp}, \smash{$\Sigmacol^{-1/2}\varepsilon$} is a matrix with independent standard normal elements,
so \smash{$z=\psi_{col}'\varepsilon = (\psi_{col}' \Sigmacol^{1/2} \Sigmacol^{-1/2}\varepsilon)$} is a vector with independent mean-zero gaussian elements 
with standard deviation \smash{$\norm{\psi_{col}'\Sigmacol^{1/2}}$}. The expectation of the supremum over $z'\delta$ for $\delta \in \sTheta_s$ is 
\smash{$\norm{\psi_{col}'\Sigmacol^{1/2}} \width(\sTheta_s)$}, and by the gaussian concentration inequality for Lipschitz functions \citep[e.g.][Theorem 5.2.2]{vershynin2018high},
the supremum itself is bounded by \smash{$\norm{\psi_{col}'\Sigmacol^{1/2}} \{\width(\sTheta_s) + us\}$} with probability 
$1-2\exp(-cu^2)$. Taking $u= w_1 \width(\sTheta_s)/s$, this bound is 
\smash{$(1+w_1) \norm{\psi_{col}'\Sigmacol^{1/2}} \width(\sTheta_s)$} and it holds
with probability $1-2\exp\{-c w_1^2 \width^2(\sTheta_s)/s^2\}$.

In the fourth term of \eqref{eq:sc-oracle-deviation-decomp}, 
\smash{$\varepsilon_e - \psi_{col}' \varepsilon$} is uncorrelated with \smash{$[\varepsilon,\nu]$} and therefore independent of it, as the two are jointly normal.
As \smash{$\htheta - \tilde \theta$} is $[\varepsilon, \nu]$-measurable, 
\smash{$(\varepsilon_e - \psi_{col}' \varepsilon)(\htheta - \stheta)$}
is conditionally normal with mean zero and standard deviation \smash{$\norm{\varepsilon_e - \psi_{col}' \varepsilon}_{L_2}\norm{\htheta - \stheta}$},
therefore bounded by \smash{$w_2 \norm{\varepsilon_e - \psi_{col}' \varepsilon}_{L_2}\norm{\htheta - \stheta}$} with probability
$1-2\exp(-w_2^2)$. 

By the union bound and Theorem~\ref{theorem:rate-simplified}, 
all bounds used above hold on an event of probability $1-c\exp\set{-cu(s,v)}-2\exp\{-c w_1^2 \width^2(\sTheta_s)/s^2\} - 2\exp(w_2^2)$.
The claimed bound holds on that event. 

\subsection{Proof of Corollary~\ref{corollary:sc-normality}}
To show normality of the z-statistic, it suffices to show that
for fixed $(v,w_1,w_2)$, the bound 
\eqref{eq:sc-deviation} from Proposition~\ref{proposition:sc-bias-bound} 
are \smash{$o(\sigma_{e}p_{eff}^{-1/2})$} and holds with arbitrarily high probability. It hold with arbitrarily high probability
if $u(s,v)$ and $\width^2(\sTheta_s)/s^2$ are bounded away from zero, and therefore if $\min(\sigma^2,1)\width^2(\sTheta_s)/s^2$ is.
We will characterize $s$ as in Example~\ref{example:l1}, using the bound $\width^2(\sTheta_s)=c\log(p)$, so this condition reduces to 
$s^2 \lesssim \min(\sigma^2, 1)\log(p)$. Characterizing $s$ as in Example~\ref{example:l1}, this holds if 
\[ \sigma^2\log(p) + \set{\sigma \norm{A\stheta - b} + \sigma^2(n/p_{eff})^{1/2}}\sqrt{\log(p)} + 
    \{\sigma^2 \rank(A)/p_{eff}\}^{1,1/2} \lesssim \eta^2 n\min(\sigma^2,1)\log(p). \]
This is equivalent to the following bounds on $\max(\sigma^2, 1)$ and $\max(\sigma, 1/\sigma)$.
\begin{equation}
\label{eq:near-vacuous-eta-bounds}
\begin{aligned}
\max(\sigma^2, 1) &\lesssim \eta^2 \min\sqb*{ n,\ \{p_{eff} n \log(p)\}^{1/2}, \ p_{eff} n \log(p) / \rank(A) }; \\
\max(\sigma, 1/\sigma) &\lesssim \eta^2 \min\sqb*{ \{n \log(p)\}^{1/2} / \delta_{fit},\ p_{eff}^{1/2} n \log(p) / \rank(A)^{1/2}}.
\end{aligned}
\end{equation}
The three bounds on the former arise from the terms $\sigma^2\log(p)$, $\sigma^2(n/p_{eff})^{1/2}\sqrt{\log(p)}$,
and $\sigma^2 \rank(A)/p_{eff}$ respectively on the left side. The two bounds on the latter arise from
the terms $\sigma \norm{A\stheta - b} \sqrt{\log(p)}$ and  $\{\sigma^2 \rank(A)/p_{eff}\}^{1/2}$.

Having shown that our bound holds with arbitrarily high probability, we will now show that it is sufficiently small.
The first term of \eqref{eq:sc-deviation} is \smash{$o(\sigma_e p_{eff}^{-1/2})$} if \smash{$(s/\sigma)\tilde D$} \smash{$\ll \sigma_{e}p_{eff}^{-1/2}$}, 
i.e., if \smash{$s^2 \ll \sigma^2 \sigma_e^2 / (p_{eff} \tilde D^2)$.}
Again using Example~\ref{example:l1} to characterize $s$, this holds if 
\[ \sigma^2\log(p) + \set{\sigma \norm{A\stheta - b} + \sigma^2(n/p_{eff})^{1/2}}\sqrt{\log(p)} + 
    \{\sigma^2 \rank(A)/p_{eff}\}^{1,1/2} \ll \frac{\sigma^2 \eta^2 n \sigma_e^2 }{ p_{eff} \tilde D^2}. \]
This is equivalent to the following bounds on $\rank(A)$ and $p_{eff}$.
\begin{equation*}
\begin{aligned}
\rank(A) &\ll \min\set*{ n \p*{\frac{\eta \sigma_e}{\tilde D^2}}^2,\ \frac{\sigma^2 n^2}{p_{eff}} \p*{\frac{\eta \sigma_e}{\tilde D}}^4 }; \\
p_{eff} &\ll \min\set*{\frac{n}{\log(p)} \p*{\frac{\eta \sigma_e}{\tilde D}}^2, \ 
		       \frac{n}{\log(p)} \p*{\frac{\eta \sigma_e}{\tilde D}}^4, \  
		       \frac{\sigma n}{\norm{A\stheta - b}\sqrt{\log(p)}} \p*{\frac{\eta \sigma_e}{\tilde D}}^2 }.
\end{aligned}
\end{equation*}
The two bounds on $\rank(A)$ arise from the terms $\sigma^2 \rank(A)/p_{eff}$
and $\{\sigma^2\rank(A)/p_{eff}\}^{1/2}$ respectively on the left side.
The three bounds on $p_{eff}$ arise from the terms $\sigma^2\log(p)$, $\sigma^2(n/p_{eff})\sqrt{\log(p)}$,
and $\sigma \norm{A\tilde\theta - b}\sqrt{\log(p)}$ respectively.

The second term of \eqref{eq:sc-deviation}, $\norm{\psi_{col}'\Sigmacol^{1/2}}\sqrt{\log(p)}$, 
is $o(\sigma_{e}p_{eff}^{-1/2})$ as well if $p_{eff}$ satisfies 
\[ p_{eff} \ll \sigma_e^2 / \{\log(p) \norm{\psi_{col}'\Sigmacol^{1/2}}^2\}. \]

All of these conditions on conditions on $p_{eff}$ and $\rank(A)$ are assumed with the exception of the bound
\[ p_{eff}  
\ll \frac{\sigma n}{\norm{A\stheta - b}\sqrt{\log(p)}} \p*{\frac{\eta \sigma_e}{\tilde D}}^2
= \frac{\sigma}{n^{-1/2}\norm{A\stheta - b}} \p*{\frac{n}{\log(p)}}^{1/2} \p*{\frac{\eta \sigma_e}{\tilde D}}^2.
  \]
We conclude by showing this is implied by our other conditions. As we have assumed that 
$\sigma_e p_{eff}^{-1/2} n^{(\kappa-1)/2} \gg n^{-1/2}\norm{A\stheta - b}$, substituting this upper bound in the denominator implies that this holds if
\[ p_{eff}^{1/2} \lesssim \frac{\sigma n^{(1-\kappa)/2}}{\sigma_e}\p*{\frac{n}{\log(p)}}^{1/2} \p*{\frac{\eta \sigma_e}{\tilde D}}^2,
\ \text{ i.e., }\ 
   p_{eff} \lesssim n^{1-\kappa}\p*{\frac{\sigma}{\sigma_e}}^2 \p*{\frac{n}{\log(p)}} \p*{\frac{\eta \sigma_e}{\tilde D}}^4. \]

The substitution of the stated singular value bound for the rank bound is justified by 
Theorem~\ref{theorem:rate-simplified}, which allows the substitution for $\rank(A)$ of any $R$ satisfying 
\[ \sigma_{R+1}(A) \le c(s+v \sigma p_{eff}^{-1/2}) R  / \width(\sTheta_s) \] 
As $s \ge 0$ and we take $\width(\sTheta_s)=c\sqrt{\log(p)}$, it follows that for large enough (constant) $v$,
this holds under the stated condition $\sigma_{R+1}(A) \le \sigma p_{eff}^{-1/2} R  / \sqrt{\log(p)}$.

\end{appendix}
\end{document}